\documentclass[a4paper, 11pt]{amsart} 

\usepackage[english]{babel}

\usepackage[numbers]{natbib}

\usepackage[margin=1in]{geometry}
\usepackage{etex}

\usepackage{amsbsy}
\usepackage{amsmath,amsfonts,amssymb,amsthm,mathrsfs,mathtools}

\usepackage[font=sf, labelfont={sf,bf}, margin=1cm]{caption}
\usepackage{graphicx}
\usepackage{epsfig}%pour les epsk0
\usepackage{latexsym}%encore des symboles
\usepackage{xcolor}
\usepackage{ae,aecompl}
\usepackage{soul,framed}
\usepackage{comment}

\usepackage[font=small]{caption}

\usepackage[colorinlistoftodos, textsize=tiny]{todonotes}
		
\usepackage{xcolor}
\usepackage[pdfpagemode=UseNone,bookmarksopen=false,colorlinks=true,urlcolor=blue,citecolor=blue,citebordercolor=blue,linkcolor=blue]{hyperref}
\usepackage{smfhyperref}	
\usepackage[capitalize]{cleveref}
\usepackage{enumitem}

\usepackage{pstricks}
\usepackage{tikz,animate,media9}						%%%%%%%%%%
\usepackage{todonotes}
\usepackage{pifont}
\usepackage{bm,marvosym}
\usepackage{algorithm}
\usepackage{algorithmic}

\usepackage{bbm}

\newcommand{\ndN}{\mathbb{N}}

\renewcommand{\Pr}[1]{\mathbb{P}(#1)}

\newcommand{\Prb}[1]{\mathbb{P}\left(#1\right)}

\newcommand{\Ex}[1]{\mathbb{E}[#1]}

\newcommand{\Exb}[1]{\mathbb{E}\left[#1\right]}

\newcommand{\Va}[1]{\mathbb{V}[#1]}

\newcommand{\one}{{\mathbbm{1}}}

\newcommand{\cA}{\mathcal{A}}
\newcommand{\cB}{\mathcal{B}}
\newcommand{\cC}{\mathcal{C}}
\newcommand{\cD}{\mathcal{D}}

\newcommand{\cG}{\mathcal{G}}
\newcommand{\cH}{\mathcal{H}}

\newcommand{\cP}{\mathcal{P}}
\newcommand{\cQ}{\mathcal{Q}}
\newcommand{\cR}{\mathcal{R}}
\newcommand{\cS}{\mathcal{S}}
\newcommand{\cT}{\mathcal{T}}

\newcommand{\cX}{\mathcal{X}}

\newcommand{\Cyc}{\textsc{CYC}}
\newcommand{\Set}{\textsc{SET}}
\newcommand{\Seq}{\textsc{SEQ}}

\newcommand{\mA}{\mathsf{A}}

\newcommand{\mC}{\mathsf{C}}
\newcommand{\mD}{\mathsf{D}}

\newcommand{\mR}{\mathsf{R}}

\newcommand{\mT}{\mathsf{T}}

\newcommand{\eps}{\varepsilon}

\newtheorem{theorem}{Theorem}[section]

\newtheorem{lemma}[theorem]{Lemma}

\newtheorem{definition}[theorem]{Definition}
\newtheorem{remark}[theorem]{Remark}
\newtheorem{example}[theorem]{Example}
\newtheorem{algo}[theorem]{Algorithm}
\newtheorem{fact}[theorem]{Fact}

\numberwithin{equation}{section}

\begin{document}

%%neu
\title{Exact-size Sampling of Enriched Trees in Linear Time} 

\address{Department of Mathematics, Ludwigs-Maximilians-Universit\"at M\"unchen.}
\author{Konstantinos Panagiotou}
\email{kpanagio@math.lmu.de}
\author{Leon Ramzews}
\email{ramzews@math.lmu.de}
\author{Benedikt Stufler}
\address{Institute for Discrete Mathematics and Geometry, Technische Universit\"at Wien.}
\email{benedikt.stufler@tuwien.ac.at}

\thanks{Konstantinos Panagiotou receives funding from the European Research Council, ERC Grant Agreement 772606-PTRCSP.
Leon Ramzews receives funding by the Deutsche Forschungsgemeinschaft (DFG, German Research Foundation), Project PA 2080/3-1.}

\begin{comment}
\thanks{Department of Mathematics, Ludwigs-Maximilians-Universit\"at M\"unchen. E-mail: kpanagio@math.lmu.de.}\, ~~ Leon Ramzews\thanks{Department of Mathematics, Ludwigs-Maximilians-Universit\"at M\"unchen. E-mail: ramzews@math.lmu.de. Funded by the Deutsche Forschungsgemeinschaft (DFG, German Research Foundation), Project PA 2080/3-1.}\, ~~ Benedikt Stufler\thanks{Institute for Discrete Mathematics and Geometry, Technische Universit\"at Wien. E-mail: benedikt.stufler@tuwien.ac.at.}
\end{comment} 

\date{\today}
\maketitle
%%neu

\begin{abstract}
	Various combinatorial classes such as outerplanar graphs and maps, series-parallel graphs, substitution-closed classes of permutations and many more allow bijective encodings by so-called \emph{enriched trees}, which are rooted  trees with additional structure on the offspring of each node. Using this universal description we develop sampling procedures that uniformly generate objects from this classes with a given size $n$ in expected time $O(n)$.
	%Under general conditions that are  satisfied by all aforementioned examples we develop a uniform sampling procedure for enriched trees running in expected linear time. 
	The key ingredient is a representation of enriched trees in terms of decorated Bienaym\'e--Galton--Watson trees, which allows us to develop a novel combination of Devroye's efficient sampler for trees~\cite{MR2888318} with Boltzmann sampling techniques. Additionally, we construct expected linear time samplers for critical Bienaym\'e--Galton--Watson trees having exactly $n$ (out of $\ge n$ total) nodes with outdegree in some fixed set, enabling uniform generation for many combinatorial classes such as dissections of polygons.
\end{abstract}

\maketitle

%###################################### Introduction #######################################
\section{Introduction and Main Results}
\label{sec:introduction}

Suppose that we are given a combinatorial class $\cC$, that is, a countable set equipped with a size function $\lvert\cdot\rvert:\cC\to \ndN_0$ such that $\cC_n := \{C\in\cC:\lvert C\rvert = n\}$ is finite for all $n\in\ndN_0$.
A \emph{sampler} for $\cC$ is a sequence of instructions involving random decisions that construct an element $\mC \in \cC$ following some given probability distribution.
The development of efficient samplers, or equivalently, the efficient random generation of combinatorial objects, is an active and prominent research area with widespread applications. There is a plethora of results and techniques, many of which address specific problems and develop ad hoc methods, and others that create universal techniques that are applicable in various situations. 

Let us start right away with an important case that is very well understood and also directly relevant to the results that will be derived here. Let $\xi$ be a random non-negative integer. We create a tree~$\mathsf{T}$ by starting with a single vertex and attaching to it a random number of vertices/children distributed like~$\xi$. Subsequently, we repeat this procedure for every newly created vertex, using independent copies of $\xi$ to determine the number of their children.
The resulting tree is the well-known \emph{Bienaym\'e--Galton--Watson} tree with offspring distribution~$\xi$, and by conditioning $\mathsf{T}$ to have $n$ vertices we obtain a simply generated tree $\mathsf{T}_n$. % taking values in the class of trees with $n$ vertices.
For example, if we choose $\xi$ to be a Poisson distribution, then the distribution of $\mathsf{T}_n$ (after distributing labels to vertices uniformly at random) is the uniform distribution on the class of all rooted Cayley trees with $n$ vertices.  Improving earlier results addressing special cases
%(for example \cite{arnold1980,MR3694485} and \cite[Ch.~XIII.5]{devroye1986non})
or having a longer running time,
%~\cite{MR2095975}
Devroye presented in~\cite{MR2888318} a general algorithm for sampling $\mathsf{T}_n$ that runs in expected linear time when $\Ex{\xi}=1$ and $\xi$ has finite variance.
So, this fundamental case is from today's viewpoint very well understood.

Probably the first systematic approach that applies to a broader setting is the \emph{recursive method} by Nijenhuis and Wilf~\cite{MR510047} that can be applied to combinatorial structures defined by specific recursive decompositions.
%Just to give an example, a rooted binary tree is either a single node or a node with two subtrees, both which are again rooted binary trees.
The original method, although quite broad in applicability, was rather inefficient and thus was developed further and improved in several works~\cite{MR1290534,DZ99}, where eventually samplers with almost linear average time and space complexity were developed. However, all these results are limited to classes that do not allow in general the powerful operation of \emph{substitution}, that is, constructions in which atoms (like vertices or edges in a graph) are replaced by other objects; this limits the applicability of the method to only moderately complex combinatorial classes. Moreover, all variants of the recursive method require (at least) quadratic preprocessing time. 

The recursive method is best suited for \emph{exact-size} sampling, where we fix, for example, the size of the object that we want to sample in advance. This paradigm was relaxed in the seminal paper by Duchon, Flajolet, Louchard and Schaeffer~\cite{MR2095975}, allowing samplers to generate objects with varying target size that may be distributed over the whole of $\mathbb{N}$. The so-called \emph{Boltzmann sampling} paradigm developed in that paper is inspired from methods in Physics and postulates to generate objects from the entire class $\cC$ with probability proportional to $x^{\lvert C\rvert}/\lvert C\rvert!$ for $C\in\cC$, where $x>0$ is a predefined control parameter. 
Boltzmann samplers have many advantages:  their complexity is (for combinatorial specifications) linear in the size of the generated object, in many cases we have good control of the output size by tuning $x$, and their description is simple and intuitive.
For all these reasons the paper~\cite{MR2095975} ignited a whole new line of research, where substantial extensions and improvements were proposed, including the celebrated approximate-size linear time sampler for planar graphs~\cite{MR2573060}, substantial P\'olya-Boltzmann extensions for unlabelled structures~\cite{MR2498128,MR2810913}, multi-parametric samplers~\cite{MR2735331,MR3773638} enabling the control of several parameters simultaneously, numerical procedures for approximating the values of the generating functions~\cite{MR2946384} and many more~\cite{MR2593621,MR3101704,MR2971340}.
The Boltzmann framework enables us to perform exact-size sampling by rejection (discard objects until the target size is met) and truncation (stop sampling as soon as the objects become too large). In particular, exact-size sampling is possible in expected quadratic time whenever the counting sequence for $(|\cC_n|)_{n \in \mathbb{N}}$ satisfies certain properties, for example if $|\cC_n| =\Theta(n^{-a} \gamma^{n} n!)$ for some $a \in (1,2)$ and $\gamma > 0$, see the so-called 'singular samplers' in~\cite{MR2095975}. The Boltzmann framework also enables sampling in this setting with a target size interval of the form $[(1- \epsilon)n, (1+\epsilon)n]$ in expected time $O(n/\epsilon)$ for fixed but arbitrary $\epsilon>0$. However, this so-called approximate size sampling may introduce unpredictable error terms when performing simulations, so there is significant added value in performing efficient exact-size sampling.
 
In this article we combine the world of Devroye -- linear time sampling of conditioned Bienaym\'e--Galton--Watson trees -- with the world of Boltzmann sampling to assemble efficient linear time and exact-size samplers that are applicable to a broader spectrum of combinatorial classes.
Concretely, the classes that we can treat follow a unified representation in terms of \emph{enriched trees}~\cite{MR3773800, StEJC2018}. Before giving a formal definition later, let us mention a few concrete examples that fall within our scope.
%, for most of which there are no (linear-time) samplers yet.
\begin{example}
\label{ex:examples_intro}
Our approach allows us to treat in a unified setting the following classes:\footnote{Implementations for  exact-size samplers of selected classes such as outerplanar graphs are available on github:~ \url{https://github.com/BenediktStufler/}}
\begin{enumerate}
    \item subcritical graph classes, including connected series-parallel, outerplanar and cactus graphs;
    \item Bienaym\'e--Galton--Watson trees conditioned on the number of vertices whose degree lies in a given set and otherwise no restriction on the size;
    \item families of outerplanar maps; 
%\item Connected subcritical block-stable classes of graphs, like series-parallel, outerplanar, block and cactus graphs;
    \item certain subcritical substitution-closed classes of permutations;
    \item cographs (with expected runtime being linear in the output size);
    \item level-$k$ phylogenetic networks.
\end{enumerate}
\end{example}
Let us describe at this point exemplary for the case of series-parallel (SP) graphs what the main obstacles in the development of efficient and exact-size samplers are. First of all, the good news is that these objects can be put in some specific way in bijection to a class of trees; this follows from the general decomposition of connected graphs in parts of higher connectivity~\cite{MR2465772}. On the other hand, these trees are not simply-generated -- in fact, they are \emph{multi-type} Bienaym\'e--Galton--Watson trees -- and thus Devroye's sampler is not applicable. Moreover, the combinatorial specification of 2-connected SP graphs, which play a central role in the specification of connected SP graphs, contains the operation of \emph{difference} of classes (see also Section~\ref{ssec:sccg}, where we treat this specific example). This is a significant obstacle causing problems on various levels of the analysis, as it introduces a '$-$'-sign on the level of the specification. A further problem that also appears (more prominently) in other examples is that sampling from the specification is only the first step: in order to obtain the desired object we further have to apply a bijection. The time required to do so must be taken into account as well. 

The approach taken in this work makes it possible to develop expected linear-time exact-size samplers for the classes in Example~\ref{ex:examples_intro} by addressing all of the aforementioned problems. Very roughly speaking, the problem of the appearance of multi-type trees is addressed by studying the class of enriched trees, that puts us in a position to spot an adequate underlying 'simply-generated' tree. Moreover, the Boltzmann sampling component that we include allows us to solve the problem of the difference of sets by rejection in a rather straightforward way. Finally, we account explicitly in all examples for the cost of realizing the underlying bijections.

One consequence of our main result is the development of new or the improvement of all (with the notable exception of outerplanar maps \cite{MR2185278}) existing samplers for the classes in Example~\ref{ex:examples_intro}.
For instance, prior to this work, the state-of-the-art samplers for series-parallel, outerplanar and cactus graphs run in expected time $O(n^2)$ and are based on Boltzmann sampling, as the number of such graphs of size $n$ is $\Theta(n^{-3/2} \gamma ^n n!)$ for some $\gamma > 0$. See~\cite{zbMATH05039060} and \cite{zbMATH07106247}. 
Moreover, the sampler for 3.~in the examples can be used to sample from classes that are in bijection to these objects, for example dissections of polygons. 
Finally, in~\cite{MR3637994}, among other results, a superlinear uniform sampler for substitution-closed classes of permutations with a finite number of excluded patterns is presented.
Here we will show how to sample from these classes in (optimal) linear time as a consequence of a more general result.

As a last remark let us mention that parallel to this work and by using completely different techniques, Sportiello developed in~\cite{sportiello2021boltzmann} a rather different method for sampling from irreducible context-free combinatorial structures. His approach solves the problem of exact-size sampling from multi-type Bienaym\'e--Galton--Watson trees and thus addresses in a different and complementary way some of the problems that are also tackled here. Moreover, the approach presented there only makes it possible to sample  from classes related to SP graphs, for example SP networks (or two-terminal SP graphs). 

\subsection{Main result: linear-time sampling of enriched trees}
\label{subsec:main_result}
When measuring the running time or the complexity of an algorithm we always assume that we operate under the so-called \emph{RAM model of computation}. This is a widely used approach, followed also in Devroye's paper~\cite{MR2888318}, to establish complexity results that are  independent of the actual machine on which the algorithm is executed. In this model we assume that we operate a hypothetical computer called the \emph{Random Access Machine}  under the following conditions, see for example \cite[Ch.~2]{Skienna2020}:
\begin{enumerate}
\item Basic logical and arithmetic operators like $\{if, call\}$ and $\{+,-,\times,\div\}$ take one time step.
\item Loops and subroutines are not basic operators and count as the composition of many single-step operators.
\item Each memory access takes exactly one step.
\item A number drawn uniformly at random from $[0,1]$ can be generated in one step.
\end{enumerate}
In particular, in this model real numbers can be stored without loss of precision. Let us mention at this point that it is an important question and a significant challenge to develop and analyse algorithms that operate efficiently under other models of computation, for example on a Word RAM with random bits or on a Turing machine; to our knowledge, this is already an open problem for the case considered by Devroye~\cite{MR2888318}.

In the setting considered here we can already formulate a first consequence of our main result, which, among other things, says that we develop linear-time algorithms for sampling series-parallel and outerplanar graphs, as well as permutations from specific substitution-closed classes and trees with a given number of leaves.
\begin{theorem} 
\label{thm:example_main_result}
Under the RAM model of computation, the samplers in Section~\ref{sec:applications} produce objects of size $n$ uniformly at random in expected time $O(n)$ for all the classes in Example~\ref{ex:examples_intro}.
\end{theorem}

Towards the proof of Theorem~\ref{thm:example_main_result} we will switch somehow our point of view and look at combinatorial classes as special cases of so-called 'enriched trees'.  
Generally speaking, the enriched tree viewpoint emphasizes that Galton--Watson trees take a special place among random recursive structures.
%They are the driving force behind many phenomena observed for a variety of models.
Specifically, we may sample a random recursive structure by sampling a size-constrained Galton--Watson tree and adding local random 'decorations' later. All classes listed in Example~\ref{ex:examples_intro} are well-known  to admit encodings of this form.

Before we continue let us fix some notation.
Let $\cC$ be a combinatorial class. We (always) consider the labelled setting, meaning that all atoms composing an object of size $n$ bear distinct labels, typically in the set $[n]:=\{1,\dots,n\}$. Any other finite set of labels $U$ with $\lvert U\rvert=n$ is also admissible; then we write $\cC[U]$ to emphasize that we consider objects with labels in~$U$.

With this notation at hand we may describe the main class of objects that we shall study. Given a combinatorial class $\cR$, the class $\cA_\cR$ of $\cR$-enriched trees may informally be described as the class containing all rooted labelled unordered trees, where in addition the offspring set of each vertex is decorated with an object from $\cR$.  More formally, let us denote  by $\cA$ the class of rooted Cayley trees, i.e., rooted labelled unordered acyclic connected graphs. We (slightly) abuse  notation and write $v \in T$ to denote that $v$ is a vertex of $T \in \cA$. Let $P_v$ be the label set of the offspring of $v$, where by 'offspring' we define the set of nodes that are connected to $v$ and are at the same time farther away from the root than $v$. We further define the outdegree $d_T^+(v)$ to be $\lvert P_v\rvert$. Then $\cA_\cR$ contains all sequences of the form
\[
	(T,(R_v)_{v\in T}),
	\quad T\in \cA ~\text{ and }~ R_v \in \cR[P_v]\text{ for all }v\in T,
\]
where the size of an $\cR$-enriched tree $A=	(T,(R_v)_{v\in T})$ is defined as  $\lvert A\rvert = \lvert T\rvert = \sum_{v\in T}\lvert R_v\rvert +1$. 
In light of this definition and in order to avoid trivial cases we assume that $|\cR_0| > 0$ (otherwise $\cA_\cR$ is empty) and $|\cR_k| > 0$ for some $k \ge 2$ (otherwise $\cA_\cR$ is equivalent to a collection of paths), see also Condition~\ref{assumption2} in Definition~\ref{def:tame} below.

%As we will see in Section~\ref{sec:applications}, all classes in Example~\ref{ex:examples_intro} (and many more) are (isomorphic to) enriched trees for specific choices of $\cR$.
%\begin{example}
%Let $\cR$ be the sequence class $\Seq$ which, for any finite set $U=\{u_1,\dots,u_k\}$ with $k\in\ndN_0$, is given by 
%\[
%	\Seq[U]
%	:= \{(u_{\pi(1)},\dots,u_{\pi(k)}): \pi \text{ is a permutation of }[k]\}.
%\]
%Then we claim that the class $\cA_\Seq$ is (isomorphic to) the class of all labelled rooted \emph{ordered} trees $\cT$. To see this, consider $(T,(R_v)_{v\in T})\in\cA_\Seq$. Any node $v\in T$ with outdegree $k\in\ndN$ has offspring bearing labels in $\{\ell_1,\dots,\ell_k\}$ where $\ell_i\in\ndN$ are all distinct for $1\le i\le k$. Then the corresponding sequence object $R_v \in\Seq[\{\ell_1,\dots,\ell_k\}]$ yields an ordering of the labels, respectively the corresponding nodes.
%\end{example}
Let us write in the sequel $r_k := \lvert \cR_k\rvert$ for $k\in\ndN_0$ and define the generating functions
\[
	\cR(x)
	:= \sum_{k\in\ndN_0}\frac{r_k}{k!}x^k 
	\quad\text{and}\quad 
	\cA_\cR(x) := \sum_{k \in \ndN} \frac{|\cA_\cR[[k]]|}{k!} x^k.
\]
The two functions are related by the important equation
\begin{equation}
\label{eq:relRAR}
	\cA_\cR(x) = x\cR(\cA_\cR(x)),
\end{equation}
see also Section~\ref{sec:preliminaries}, where we present more related facts about $\cR$-enriched trees. 
Let $\rho_{\cR}$ and $\rho_{\cA_\cR}$ be the radii of convergence of $\cR(x)$ and $\cA_\cR(x)$, respectively. It is rather well-known that Equation~\eqref{eq:relRAR} and the aforementioned assumptions $r_0>0$ and $r_k >0$ for at least one $k \ge 2$ entail that $\rho_{\cA_\cR}$, $\cA_\cR(\rho_{\cA_\cR})$, and $\cR(\cA_\cR(\rho_{\cA_\cR}))$ are finite, see Lemma~\ref{lem:xi_finite_exp_mom} and for more background Section~\ref{subsec:simply_generated_trees}. 

The previous considerations imply that $\cR(\cA_\cR(\rho_{\cA_\cR})) \ge r_0 > 0$ and thus  enable us to define a random variable $\xi$ with distribution 
\begin{align}
\label{eq:def_offspring_distribution_R-enriched}
	p_k 
	:= \Pr{\xi = k}
	:= \frac{r_k\cA_\cR(\rho_{\cA_\cR})^k}{ \cR(\cA_\cR(\rho_{\cA_\cR}))k!},
	\qquad k\in \ndN_0.
\end{align}
We will also need the Boltzmann random variable $\Gamma\cR(t)$ for $0<t<\rho_\cR$ given by 
\[
	\Pr{\Gamma\cR(t) = R}
	= \frac{t^{|R|}}{\cR(t) \, |R|!},
	\quad R\in \cR.%, k\in\ndN_0.
\]
We now come to the most crucial part, namely the assumptions that we make for the class of enriched trees that we consider. 
\begin{definition}
\label{def:tame}
We call a class of enriches trees \emph{tame} if it has the following properties:
\begin{enumerate}[label=(\Alph*)]
\item \label{assumption2}'Non-triviality': $r_0>0$ and there exists $k\ge 2$ with $r_k>0$. 
\item \label{assumption1}'Subcriticality': $\rho_{\cR} > \cA_\cR(\rho_{\cA_\cR})$.
\item \label{assumption3}'Computability':  $\cA_\cR(\rho_{\cA_\cR})$ and $\cR(\cA_\cR(\rho_{\cA_\cR}))$ are given, and $p_k$ can be computed in $\mathrm{e}^{o(k)}$ steps for $k\in\ndN_0$.
\item \label{assumption4}'Boltzmann sampler for $\cR$': For any $0<t<\rho_{\cR}$ there exists a sampling procedure for $\Gamma \cR(t)$ running in expected constant time.
\end{enumerate}
\end{definition}
As we will see, the most critical and essential property is Assumption~\ref{assumption1}, which ensures together with well-known results by Janson~\cite{MR2908619} that $p_k$ has exponential tails. 
Moreover, note that under the RAM model of computation the determination of $\cA_\cR(\rho_{\cA_\cR})^k/k!$ takes $O(k)$ steps, so that~\ref{assumption3} actually means that we need to be able to compute $r_k$ (which is in $\ndN_0$) in $\mathrm{e}^{o(k)}$ steps. This is usually no severe restriction, as most of the classes we consider have some kind of combinatorial decomposition, allowing us to use the aforementioned recursive method~\cite{MR510047} to compute $r_k$ in $k^{O(1)}$ time. Further, $\cA_\cR(\rho_{\cA_\cR})$ and $\cR(\cA_\cR(\rho_{\cA_\cR}))$ being given means that we are able to compute these values beforehand.
%This is sometimes referred to as the oracle assumption. As the precision of real numbers does not concern the combinatorial complexity of algorithms we do not delve deeper into this matter.
Finally,~\ref{assumption4} is a rather mild condition, since $t < \rho_\cR$ and thus all moments of $\Gamma \cR(t)$ exist; such samplers can (usually) be designed from the general principles developed in~\cite{MR2095975,MR2810913}. 

A crucial ingredient in the proof of Theorem~\ref{thm:example_main_result} is the following fact on which we elaborate in Section~\ref{sec:applications}.
\begin{fact}
\label{fact:examples_enriched_trees}
For every combinatorial class $\cC$ given in Example~\ref{ex:examples_intro} there exists $\cR$ such that the class~$\cC$ corresponds to the class of enriched trees $\cA_\cR$, where $\cA_\cR$ is tame. 
\end{fact}
In Section~\ref{sec:algorithm} we present the backbone of our main result, namely a sampler generating an instance of the random object $\mathsf{A}_n$ drawn uniformly at random from all objects in $\cA_\cR$ of size $n$. Hence together with the next theorem Fact~\ref{fact:examples_enriched_trees} guarantees that the linear time samplers claimed in Theorem~\ref{thm:example_main_result} do indeed exist.
\begin{theorem}
\label{thm:algo_runtime_n}
Assume $\cA_\cR$ is tame. Then under the RAM model of computation the sampler in Section~\ref{sec:algorithm} generates the $\cR$-enriched tree $\mathsf{A}_n$ in expected time $O(n)$.
\end{theorem}
At this point we already anticipate that the algorithm generating $\mathsf{A}_n$ is structured as follows. We first sample a Bienaym\'e--Galton--Watson tree with offspring distribution $\xi$ of size $n$ using Devroye's algorithm. Subsequently, for each node we repeatedly call the sampling procedure $\Gamma\cR(t)$ until an $\cR$-object of the same size as the outdegree of the node at hand is produced. The latter step enhances the offspring of each node with an additional structure leading to an $\cR$-enriched tree.

\subsection{Plan of the paper}
In Section~\ref{sec:algorithm} we present our sampler, Algorithm~\ref{algo:R-enriched_tree}, for tame $\cR$-enriched trees as claimed in Theorem~\ref{thm:algo_runtime_n}. To prove the linear time complexity of our algorithm we first collect some preliminaries in Section~\ref{sec:preliminaries}. In particular, we recall basic facts about combinatorial classes and $\cR$-enriched trees, simply generated trees and local limit theorems for iid random variables. Subsequently, the proof of Theorem~\ref{thm:algo_runtime_n} is conducted in Section~\ref{sec:proof}. Finally, in Section~\ref{sec:applications} we explain in detail how the sampler for $\cR$-enriched trees can be used to obtain linear time samplers for the classes listed in Example~\ref{ex:examples_intro}. We emphasize that for most of the examples this is not 'just' an application of Algorithm~\ref{algo:R-enriched_tree} but a rather involved task leading to new efficient sampling procedures for the classes at hand.

%As we want to focus on the combinatorial complexity of the sampling algorithms we ignore the complexity of translating between isomorphisms, that is the representation in terms of $\cA_\cR$ to other probably better-known forms.
\section{The Sampler}
\label{sec:algorithm}
In this section we present the sampler for $\mathsf{A}_n$ that is needed in the proof of Theorem~\ref{thm:algo_runtime_n}.
We briefly recall the  construction of a Bienaym\'e--Galton--Watson (BGW) tree with offspring distribution~$\xi$.
We start with a distinguished root to which a number of ordered children according to an independent copy of $\xi$ is appended.
Repeat this procedure for any node that has not received any children yet or the outcome of the copy of $\xi$ was $0$. The result is the arbitrarily sized \emph{unlabelled ordered} rooted tree $\mT$ such that the distribution of its vertex-degrees is $(p_k)_{k\ge0}$. The corresponding size-constrained tree is defined as $\mT_n := (\mT \mid \lvert\mT\rvert =n)$ for $n\in\ndN$. We use the notation that $v\in T$ is some node in the unlabelled ordered tree $T$. With this at hand, our sampler operates as follows.

\begin{algo} (Uniform $\cR$-enriched tree of size $n$.)
\label{algo:R-enriched_tree}
\begin{enumerate}
\item Generate the size-constrained Bienaym\'e--Galton--Watson tree $\mT_n$ with offspring distribution $\xi$ as in~\cite{MR2888318}.
\item For some $\cA(\rho_{\cA_\cR}) < t_0 <\rho_{\cR}$ repeatedly call for each $v\in \mT_n$ the sampler $\Gamma \cR(t_0)$ until it produces an object $\mR_v$ of size $d_{\mT_n}^+(v)$.
\item Distribute labels in $\{1,\dots,n\}$ uniformly at random and drop the ordering of the vertices afterwards.
\end{enumerate}
\end{algo}
The last step requires some explanation.
In Steps 1 and 2 we generate an object $(T, (R_v)_{v\in T})$, where $T$ is an unlabelled ordered rooted tree of size $n$ and $R_v \in \cR[[d^+(v)]]$ for $v\in T$.
The ordering of the offspring of $v\in T$ corresponds to a canonical labelling of the vertices, say $(v,1),\dots,(v,d^+(v))$ so that we may see the object $R_v$ as being labelled with elements from $\{(v,1),\dots,(v,d^+(v))\}$. %(simply consider the bijective mapping $i\mapsto(v,i)$ for $i\in[d^+(v)]$)
Naming the root $o$, the children of the root are hence labelled $(o,1),\dots,(o,d^+(o))$, the children of the first child of the root by $((o,1),1), \dots, ((o,1), d^+(o,1))$ and so on.
In particular the labels are all distinct and by generating a uniform permutation of $[n]$ we may canonically relabel the entire object with labels in $[n]$.

\begin{remark}
Alternatively we could use in Step (2) of the algorithm an \emph{exact-size} sampler for obtaining objects from $\cR_{k}$ for $k=d^+_{\mathsf{T}_n}(v)$ that runs in time $e^{ck}$ for some (small) $c > 0$. This will not affect the asymptotic expected running time, as we will see later that typically the largest degree in $\mathsf{T}_n$ is in $O(\log n)$. However, we will not consider this setting further here.
\end{remark}

The next lemma guarantees that Algorithm~\ref{algo:R-enriched_tree} produces an uniform object of size $n$ from~$\cA_\cR$. 
\begin{lemma}
	\label{le:disenrichedtree}
In distribution $(\mT_n,(\mR_v)_{v\in \mT_n})= \mathsf{A}_n$. 
\end{lemma}
The proof is rather straightforward (combine the distribution of $\mT_n$ with the fact that a Boltzmann sampler generates objects of a given size uniformly) and can be found in~\cite[Lem.~6.1]{MR4132643}. 
So, the proof of Theorem~\ref{thm:algo_runtime_n} boils down to validating that Algorithm~\ref{algo:R-enriched_tree} can be implemented to run in expected linear time.
This will be done in Section~\ref{sec:proof}.
Before we come to that, let us  first have a closer look at Devroye's algorithm~\cite{MR2888318} for sampling size-constrained trees, that is, Step~1 of Algorithm~\ref{algo:R-enriched_tree}. Any rooted ordered tree is uniquely determined by its outdegree sequence in breadth first search order and, on the other hand, any sequence $(d_1,\dots,d_n)$ such that $\sum_{1\le i\le n}d_i=n-1$ and $1+\sum_{1\le i\le t}(d_i-1)>0$ for all $1\le t\le n$ corresponds uniquely to such a tree. Define $S_t:=1+\sum_{1\le i \le t}(\xi_i-1)$ for $1\le t\le n$. The outdegrees of a Bienaym\'e--Galton--Watson trees are per definition given by the offspring distribution $\xi$, so that the random tree $\mT_n$ can be identified with the distribution of
\begin{align}
	\label{eq:condition_degree_sequence_GW}
	(\xi_1,\dots,\xi_n) ~~\bigg|~~ \left\{
	\sum_{1\le i\le n}\xi_i = n-1,
	S_t>0 \text{ for all }1\le t\le n-1 \right\}.
\end{align}
This fact, that we also shall exploit, is used in~\cite{MR2888318} to  generate efficiently $\mT_n$ as explained in the following algorithm. Recall that $p_k = \Pr{\xi = k}$.
\begin{algo}{(Size-constrained BGW tree $\mT_n$ with offspring distribution~$\xi$.)}
\label{algo:size-constrained-gw-tree-devroye}
\begin{enumerate}
    \item Sample the multinomial random vector $(N_0,N_1,\dots)$ with parameters $(n,p_0,p_1,\dots)$ repeatedly until $\sum_{1\le i\le n} iN_i = n-1$.
    \item Create a sequence of length $n$ populated with $N_j$ times $j$ for $0\le j\le n$.
    \item Randomly permute this sequence with each permutation equiprobable.
    \item Shift the elements of the sequence cyclically until the condition in~\eqref{eq:condition_degree_sequence_GW} is fulfilled.
\end{enumerate}
\end{algo}
For generating a multinomial vector in the first step \cite{MR2888318} proposes a sub-routine that samples binomial random variables.
\begin{algo}{(Multinomial random vector $(N_0,N_1,\dots)$ with parameters $(n,p_0,p_1,\dots)$.)}
\label{algo:multinomial}
\begin{enumerate}
    \item Let $N_0=\mathrm{Bin}(n,p_0)$.
    \item For $i \ge 1$, if $\sum_{0\le j\le i-1}N_j<n$, let $N_i = \mathrm{Bin}(n-\sum_{0\le j\le i-1}N_j,p_i/(1-\sum_{0\le j\le i-1}p_j)$ and otherwise set $N_i=0$.
\end{enumerate}
\end{algo}
In \cite{MR2888318} it is established that under the RAM model of computation the expected number of steps taken by Algorithm~\ref{algo:size-constrained-gw-tree-devroye} is $O(n)$, provided that $\xi$ has finite variance and that it takes one step to generate an independent copy of $\xi$. Our setting, however, is slightly different, as we do not a priori know the entire vector $(p_0,p_1,\dots)$. We need to incorporate the time it takes to compute this vector into the runtime of Step 1 of Algorithm~\ref{algo:R-enriched_tree}. Conveniently, it is sufficient to consider $(p_0,p_1,\dots)$ truncated at
$
    K := \max_{1\le i\le n}\xi_i,
$
the step at which $N_j=0$ for all $j>K$ in Algorithm~\ref{algo:multinomial} and hence the point in time after which the precise value of $p_j$ for $j>K$ is not needed anymore.
Since $\xi$ has finite exponential moments we will essentially obtain that $K = O(\log n)$ and  together with~\ref{assumption3} this will not spoil the overall linear runtime.
\begin{lemma}
\label{lem:devroye_lin_time}
If $\cA_\cR$ is tame, Algorithm~\ref{algo:size-constrained-gw-tree-devroye} can be implemented to have an expected running time of  $O(n)$.
\end{lemma}

%###################################### Preliminaries ######################################
\section{Preliminaries}
\label{sec:preliminaries}

\subsection{Combinatorial classes}
\label{ssec:classes}

In this section we recall some background information about combinatorial classes.
A comprehensive survey of the theory is given in the excellent books \cite{MR2483235,MR1629341}.
As already said, a combinatorial class is given by a countable set $\cC$ equipped with a size function $\lvert\cdot\rvert :\cC\to\ndN_0$ such that $\cC_n := \lvert \{ C\in\cC : \lvert C\rvert = n\}\rvert$ is finite for all $n\in\ndN_0$. Elements of $\cC$ are called objects or structures and any object in $\cC_n$ is said to be comprised of $n$ atoms or to be of size $n$.  We call $\cC$ labelled if each atom of an object in $\cC$ bears a distinct label.
For convenience, let the label set of any $C\in\cC_n$ be given by $[n]$.
If we want to stress out a different labelling, we simply write $C\in\cC[U]$ describing an object $C\in\cC_{\lvert U\rvert}$ labelled according to the finite set $U$. For any bijection $\sigma:U\to V$ between finite sets $U,V$ and $C\in\cC[U]$ we write $\sigma.C$ for the object obtained by replacing the label $\ell_v\in U$ of each atom $v$ of $C$ by $\sigma(\ell_v)\in V$. Clearly the resulting object is in $\cC[V]$. For coherence we assume that $C\in\cC[U]$ implies that $\sigma.C\in\cC[U]$ for any bijection $\sigma:U\to U$. The (exponential) generating series of $\cC$ is the formal power series defined by
\[
	\cC(x)
	:= \sum_{n\ge 0} \lvert\cC_n\rvert \frac{x^n}{n!}.
\]
\begin{example}
\label{ex:cayley_trees}
Consider the class $\cA$ of rooted labelled unordered acyclic connected graphs, in short Cayley trees.
We see in Figure~\ref{fi:tree_relabel} some $T\in\cA_5$ with labels in $[5]$.
Applying the bijection $\sigma$ mapping $1 \mapsto a$, $2 \mapsto b$ and so on, we retrieve an object $\sigma.T$ in $\cA[\{a,b,c,d,e\}]$.
\begin{figure}[h]
\centering
  \def\svgwidth{0.6\columnwidth}
	\resizebox{0.6\textwidth}{!}{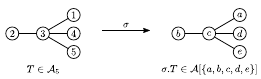}
		\caption{Relabelling $T$ under the bijection $\sigma$.}
		\label{fi:tree_relabel}
\end{figure}
\end{example}

\subsubsection{Basic Classes}
The basic combinatorial classes are the empty class $\emptyset$, the atomic class $\cX$, the set class  $\Set$, the cycle class $\Cyc$ and the sequence class $\Seq$. Let $\cS_n$ be the symmetric group containing all permutations of $[n]$ and write $(a_1,\dots,a_n)\simeq(b_1,\dots,b_n)$ if one of the two sequences is obtained by cyclically shifting the indices of the other. Then the basic classes are defined by, letting $n\in\ndN_0$,

\begin{itemize}
\item $\lvert\emptyset_n\rvert = \mathbbm{1}_{n=0}$, 
\item $\lvert\cX_n\rvert = \mathbbm{1}_{n=1}$,
\item $\Set_n =\{\{1,\dots,n\}\}$,
\item $\Cyc_n = \{(\sigma(1),\dots,\sigma(n))_\simeq: \sigma\in\cS_n\}$ and
\item $\Seq_n = \{(\sigma(1),\dots,\sigma(n)) : \sigma\in\cS_n\}$.
\end{itemize}
The respective generating series are computed to be
\[
	\emptyset(x) = 1,
	\quad
	\cX(x) = x,
	\quad
	\Set(x) = \exp(x),
	\quad
	\Cyc(x) = \log \frac{1}{1-x}
	\quad\text{and}\quad
	\Seq(x) = \frac{1}{1-x}.
\]

\subsubsection{Constructions}

Given classes $\cA$ and $\cB$ there are several ways to construct more complex classes.

\subsubsection*{Pointing}
The collection of elements $(A,a)$ where $A\in\cA$ and $a$ is a distinguished atom (the root) in $A$ forms the pointed class $\cA^\bullet$. Since $\lvert\cA^\bullet_n\rvert = n \lvert\cA_n\rvert$ we obtain
\[
    \cA^\bullet(x)
    = x \frac{\partial}{\partial x} A(x).
\]
\subsubsection*{Derivation}
Let $\star$ denote a special label such that any atom bearing this label does not contribute to the total size of the object at hand. The derived class $\cA'$ is given by the collection of objects in $\cA$ where the largest label is replaced by $\star$, i.e. $\cA'_{n-1} := \cA[\{1,\dots,n-1,\star\}]$ for all $n\in\ndN$. Hence the generating series is computed to be
\[
    \cA'(x)
    = \frac{\partial}{\partial x}\cA(x).
\]
\subsubsection*{Disjoint union} 
If $\cA$ and $\cB$ are disjoint then the disjoint union $\cA + \cB$ is the union $\cA\cup\cB$ in the standard set-theoretic sense. Formally and to avoid the assumption that the two classes at hand are disjoint we introduce two disjoint sets, say $\{0\}$ and $\{1\}$, and set
\[
    \cA+\cB
    := (\{0\}\times \cA)\cup (\{1\}\times\cB).
\]
The size of an object in $\cA+\cB$ is the size of the respective object in $\cA$ or $\cB$. It is straightforward that
\[
	(\cA+\cB)(x) = \cA(x)+\cB(x).
\]
\subsubsection*{Product} 
The product class $\cA\cdot\cB$ contains all tuples $(A,B)$ with $A\in\cA$ and $B\in\cB$ relabelled with labels in $\{1,\dots,\lvert A\rvert + \lvert B\rvert\}$. The size function is defined by $\lvert(A,B)\rvert = \lvert A\rvert + \lvert B\rvert$. The generating series is
\[
	(\cA\cdot\cB)(x) = \cA(x)\cB(x).
\]
\subsubsection*{Substitution}
For this construction we assume $\cB_0 = \emptyset$. An object in the substitution class $\cA\circ\cB$, sometimes also $\cA(\cB)$, is comprised of an $\cA$-object whose atoms are replaced by $\cB$-objects.
In other words, $\cA\circ\cB$ contains all equivalence classes of sequences of the form
\[
	(A, B_1, \dots, B_k)_\simeq,
	\quad A\in\cA[\{P_1,\dots,P_k\}], B_i \in \cB[P_i], 1\le i\le k,
\]
where $k\ge 0$ and $P_1,\dots ,P_k$ is a partition of $\{1,\dots,\sum_{1\le i\le k}\lvert B_i\rvert\}$ with $\lvert P_i\rvert = \lvert B_i\rvert$ for $1\le i\le k$. The equivalence relation ``$\simeq$" terms two sequences $(A, B_1, \dots, B_k)$ and $(A', B_1', \dots, B_k')$ isomorphic if $A=A'$ and for any permutation $\sigma:\{1,\dots,k\} \to \{1,\dots,k\}$ such that $\sigma.A = A$ it holds that $B_{\sigma(i)} = B_i'$ for $1\le i\le k$. Hence any $M\in \cA\circ\cB$ possesses a core structure $A$ and \emph{components} $B_1,\dots,B_k$. The size is then given by $\lvert M\rvert := \sum_{1\le i\le k}\lvert B_i\rvert$. The respective generating series fulfils
\[
	(\cA\circ\cB)(x)
	= \cA(\cB(x)).
\]
\subsubsection{$\cR$-enriched trees}
The theory of this subsection is extensively treated in~\cite[Ch.~3]{MR1629341}. Let $\cR$ be a combinatorial class and $\cA$ the class of rooted Cayley trees. As we already defined in Section~\ref{subsec:main_result}, the class of $\cR$-enriched trees $\cA_{\cR}$ contains all sequences
\[
	(T, (R_v)_{v\in T}),
	\qquad T \in \cA, R_v \in\cR[P_v], v\in T, 
\]
where $P_v$ is the label set of the offspring of vertex $v\in T$ and the sequence $(R_v)_{v\in T}$ is canonically ordered, for example in breath first search appearance of $v\in T$. Further, recall that the size of any $A=(T, (R_v)_{v\in T})\in\cA_\cR$ is defined as $\lvert A\rvert = \lvert T\rvert$.
%Let $\ell$ be the label of the root of $T$. Then $\{\ell\}\cup\{P_v : v\in T\}$ is a partition of $\{1,\dots,n\}$.
As any $\cR$-enriched tree is comprised of a root to which an $\cR$-structure is attached whose atoms are replaced by $\cR$-enriched trees, we obtain the functional equation
\[
	\cA_{\cR}
	= \cX \cdot \cR(\cA_{\cR}),
\]
see Theorem~2 in~\cite[Ch.~3]{MR1629341}.
This is a combination of the product and substitution construction and so the generating series satisfies the equation 
\begin{align}
\label{eq:gen_series_r-enriched}
	\cA_{\cR}(x)
	= x\cR(\cA_{\cR}(x)).
\end{align}
%Note that $\cA_\cR$ consists at least of the root in $\cX$ so that the substitution construction is well-defined.
\begin{example}
\label{ex:plane_ordered_from_enriched}
Letting $\cR$ be one of the basic combinatorial classes, we retrieve basic models of trees, see also Figure~\ref{fi:R_seq_cyc}. To wit:
%Note that we consider the labelled setting which may feel a bit unnatural when considering trees embedded in the plane since nodes are on the one hand identified by their ordering and in addition by the labels they bear. For example, rooted ordered trees possess no non-trivial automorphisms which is why every unlabelled tree of size $n$ leads to $n!$ labelled trees. But for the sake of understanding the construction of enriched trees, we deem the following simple examples to be helpful.
\begin{itemize}
    \item By choosing $\cR = \Set$ no additional structure is imposed on the offspring set of each vertex, so that we obtain that $\cA_{\Set} = \cA$, the class of rooted Cayley trees.
    \item When $\cR = \Seq$ the offspring of each vertex is given an ordering and we obtain that $\cA_{\Seq} = \cT$, the class of rooted labelled ordered trees.
    \item By cyclically ordering the offspring of each vertex (that is, $\cR=\Cyc$) we obtain the class $\cA_{\Cyc} = \cP$ of rooted labelled plane trees.
\end{itemize}
\begin{figure}[h]
\centering
  \def\svgwidth{0.6\columnwidth}
	\resizebox{0.8\textwidth}{!}{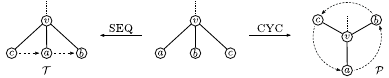}
		\caption{An $\cR$-enriched tree observed locally at some node $v$ and its offspring labelled by $\{a,b,c\}$ paired with the $\cR$-object $R_v$, where  $R_v = (b,a,c)_\simeq \in\Cyc[\{a,b,c\}]$ or $(c,b,a)\in\Seq[\{a,b,c\}]$.}
		\label{fi:R_seq_cyc}
\end{figure}
\end{example}

%Many more examples will be discussed in Section~\ref{sec:applications}.

\subsection{Simply generated trees}
\label{subsec:simply_generated_trees}
In the following we recall results concerning simply generated trees discussed thoroughly in~\cite[Ch.~7 and 8]{MR2908619}. 
We will use these results to study the offspring distribution defined in~\eqref{eq:def_offspring_distribution_R-enriched}.
Denote the class of \emph{rooted unlabelled ordered trees} by $\widetilde{\cT}$ that is obtained by taking equivalence classes under relabelling in~$\cT$, the labelled version shown in Example~\ref{ex:plane_ordered_from_enriched}.
Let $\mathrm{\omega}=(\omega_k)_{k\ge 0}$ be a real-valued non-negative sequence and set $\Phi(x):=\sum_{k\ge 0}\omega_kx^k$. For any given finite ordered tree $T$ in $\widetilde{\cT}$ define its weight by
\[
	\omega(T)
	:= \prod_{v\in T} \omega_{d^+(v)}.
\]
Define by $\mT_n(\omega)$ the $n$-sized random tree with distribution
\begin{align}
\label{eq:T_n_simply_generated}
	\Pr{\mT_n(\omega) = T}
	= \frac{\omega(T)}{Z_n},
	\quad\text{ where } T\in\widetilde{\cT}_n, ~ Z_n
	= \sum_{T'\in\widetilde{\cT}_n}\omega(T').
\end{align}
We only consider values for $n$ such that $Z_n$, the so-called partition function, is strictly greater than zero. In the next lemma we gather important properties of the generating function of the sequence of partition functions $Z(x):=\sum_{n\ge 1}Z_nx^n$. 
\begin{lemma}[{\cite[Rem.~3.2, Thm.~7.1, Rem.~7.5]{MR2908619}}]
\label{lem:roc_Z_tau}
The generating function of the partition functions satisfies the relation
\[
    Z(x) = x\Phi(Z(x)).
\]
If in addition $\omega$ is such that $\omega_0>0$ and $\omega_k>0$ for some $k\ge 2$, then $Z(x)$ has radius of convergence 
\begin{align*}
	\rho_Z = \frac{\tau}{\Phi(\tau)},
	\quad\text{where}\quad
	\tau=Z(\rho_Z) \in (0, \infty), \Phi(\tau)\in(0,\infty). 
\end{align*}
\end{lemma}
If $\omega$ is a probability weight sequence we readily notice that $\mT_n(\omega)$ defined in~\eqref{eq:T_n_simply_generated} is just the size-constrained Bienaym\'e--Galton--Watson tree with offspring distribution $\omega$.
Thus, simply generated trees are a generalization of Bienaym\'e--Galton--Watson trees. However, as we will see in a moment, using a technique called \emph{tilting}, it is often possible to view $\mT_n(\omega)$ as a Bienaym\'e--Galton--Watson tree even if $\omega$ is not a probability sequence.
We actually claim even more: there are cases -- in particular the ones that we consider here -- where it is possible to transform $\omega$ to a probability sequence such that simultaneously the distribution of the underlying simply generated tree is not altered and in addition the offspring distribution is critical, meaning that the mean is $1$ and of finite variance. More specifically, letting $\tau$ be as in Lemma~\ref{lem:roc_Z_tau}, the candidate for the offspring distribution is the \emph{tilted} sequence
\begin{align}
\label{eq:pi_k_prob_seq}
	\pi_k
	:= \frac{\omega_k\tau^k}{\Phi(\tau)},
	\quad
	k \in\mathbb{N}_0.
\end{align}
Note that in general this is not necessarily a probability sequence (for example, if $\tau = \infty$). In order to establish conditions under which $(\pi_k)_{k\ge 1}$ is a probability sequence we need some more notation.
Denote the radius of convergence of $\Phi(x)$ by $\rho$. Set
\[
	\Psi(x)
	:= \frac{x\Phi'(x)}{\Phi(x)}
	\qquad\text{and}\qquad
	\nu 
	:= \lim_{x\to\rho} \Psi(x)
	\in (0,\infty].
\]
\begin{lemma}[{\cite[Thm.~7.1, Rem.~7.5 and Ch.~8 Ia]{MR2908619}}]
\label{lem:pi_k_fin_exp_moments}
Let $\omega$ be a weight sequence such that $\omega_0>0$ and $\omega_k>0$ for some $k\ge 2$. If in addition $0<\tau<\rho$ (case $(Ia)$ in \cite{MR2908619}) then
$(\pi_k)_{k\ge 0}$ is a probability sequence with mean $1$ and finite exponential moments.
In particular we have that $\rho_Z$, $Z(\rho_Z)$ and $\Phi(Z(\rho_Z))$ are finite.
\end{lemma}
The other ranges for $\tau$ can be treated as well, see~\cite[Ch.~8]{MR2908619}, but that is not needed here.

\subsection{Local limit theorems}
Let in this subsection $\xi$ be a non-negative integer-valued random variable. Further let $\xi_1,\xi_2,\dots$ be iid copies of $\xi$ and set $\Xi_n:=\xi_1+\cdots+\xi_n$ for $n\in\ndN$. Define the span $d$ by
\[
	d
 	:= \max\{d\ge 1: d\mid i\text{ whenever }\Prb{\xi=i}>0\}.
\] 
The following local central limit theorem holds for any distribution with finite variance, so in particular in our intended application. 
\begin{lemma}[{\cite[Lem.~4.1 and Rem.~14.2]{MR2908619}}]
\label{lem:local_central_limit_theorem}
If $\Ex{\xi}=1$ and $\Va{\xi}<\infty$
\[
	\Prb{\Xi_n = n-1}
	\sim \frac{d}{\sqrt{2\pi\sigma^2n}}
	\quad \text{ as }n\to\infty\quad\text{and }n\text{ divisible by }d.
\]
Furthermore,
$$	\Prb{\Xi_n = m}
	\le \frac{d+o(1)}{\sqrt{2\pi\sigma^2 n}},
	\quad m\in\ndN.$$
\end{lemma}

%###################################### Proof ##############################################
\section{Proofs}
\label{sec:proof}
In this section we prove Theorem~\ref{thm:algo_runtime_n}. We first  show that Algorithm~\ref{algo:size-constrained-gw-tree-devroye} runs in linear time in our setting as claimed in Lemma~\ref{lem:devroye_lin_time}.  The following statement about the distribution $\xi$ defined in~\eqref{eq:def_offspring_distribution_R-enriched} will be a helpful tool for that matter and is a straightforward consequence of results about simply generated trees in presented in the previous section.  
\begin{lemma}
\label{lem:xi_finite_exp_mom}
We have $\Ex{\xi}=1$ and there exists $\eps>0$ such that $\Ex{(1+\eps)^\xi}<\infty$. In particular, the quantities $\rho_{\cA_\cR}$, $\cA_\cR(\rho_{\cA_\cR})$ and $\cR(\cA_\cR(\rho_{\cA_\cR}))$ are finite.
\end{lemma}
\begin{proof}
Let $\mT_n(\omega)$  be the simply generated tree of size $n$ with weight sequence $\omega=(r_k/k!)_{k\ge 0}$ as defined in~\eqref{eq:T_n_simply_generated}, i.e.,
\begin{align}
\label{eq:T_n_simply_generated_proof}
	\Pr{\mT_n(\omega) = T}
	= \frac{\omega(T)}{Z_n}, \qquad \text{where }
	\omega(T) = \prod_{v\in T}\frac{r_{d^+(v)}}{d^+(v)!}, ~
	Z_n = \sum_{T' \in \widetilde{\cT}_n} \omega(T'), ~
	T\in\widetilde{\cT}_n.
\end{align} 
Then the generating function $Z(x)$ of the partition functions $(Z_n)_{n\ge 0}$ is recursively given by $Z(x)=x\cR(Z(x))$. 
Let $(\cA_\cR)_n$ be the collection of objects in $\cA_\cR$ of size $n$. Since $\cA_\cR=\cX \cdot \cR(\cA_\cR)$ we know by~\eqref{eq:gen_series_r-enriched} that
$
	\cA_\cR(x)
	= x \cR(\cA_\cR(x))
$
by which we deduce 
\begin{align}
\label{eq:A(x)_equal_part_fct}
	\cA_\cR(x) = Z(x)
	\qquad\text{and}\qquad
	\frac{\lvert(\cA_\cR)_n\rvert}{n!}
	= Z_n,\quad n\in\ndN_0. 
\end{align}
We conclude that due to Lemma~\ref{lem:roc_Z_tau} the radius of convergence $\rho_{\cA_\cR}$ of $\cA_\cR(x)$ (or equivalently $Z(x)$) is given by
\[
	\rho_{\cA_\cR}
	= \frac{\tau}{\cR(\tau)},
	\qquad\text{where }
	\tau = \cA_\cR(\rho_{\cA_\cR}).
\]
Assumption~\ref{assumption1} immediately implies that $0<\tau < \rho_\cR$ and we deduce due to Lemma~\ref{lem:pi_k_fin_exp_moments} that $(\pi_k)_{k\ge 0}$ as defined in~\eqref{eq:pi_k_prob_seq} is the probability distribution with mean $1$ and finite exponential moments such that the distribution of $\mT_n(\omega)$ is not altered by switching to $(\pi_k)_{k\in\ndN_0}$. Further we observe that this is also the distribution of $\xi$ given in~\eqref{eq:def_offspring_distribution_R-enriched}, i.e.
\[
	\pi_k
	= \frac{r_k\tau^k}{\cR(\tau)k!}
	= \frac{r_k\cA(\rho_\cA)^k}{\cR(\cA(\rho_\cA))k!}
	= p_k,
	\quad k\in\ndN_0
\]
and hence the claim is verified.
\end{proof}
With these considerations at hand we first prove Lemma~\ref{lem:devroye_lin_time} and then Theorem~\ref{thm:algo_runtime_n}. In the following let $\xi_1,\xi_2,\dots$ be independent copies of $\xi$.
\begin{proof}[Proof of Lemma~\ref{lem:devroye_lin_time}]
Define 
\[
    K
    := \max_{1\le i\le n}\xi_i, \quad
    \tau_n := \Ex{K} \quad\text{and}\quad
    \varphi_n
    := \mathbb{P}\bigg(\sum_{1\le i\le n}\xi_i = n-1\bigg).
\]
According to~\cite{MR2888318} the expected time needed in Step 1 of Algorithm~\ref{algo:size-constrained-gw-tree-devroye} (i.e.~the expected running time of Algorithm~\ref{algo:multinomial}) is $O\big((1+\tau_n)/\varphi_n)\big) + O(n)$ if the probabilities $(p_0,p_1,\dots)$ are given.
Here the term $1+\tau_n$ corresponds to  the expected time until a multinomial vector $(N_0,N_1,\dots,N_K,0,\dots)$ is generated, and it takes on average $\varphi_n^{-1}$ rejections until a vector is found such that $\sum_{1\le i\le n}iN_i = n-1$.
Note that after the (random) multinomial vector has been generated in Algorithm~\ref{algo:size-constrained-gw-tree-devroye}, Step~$2$ is deterministic and takes $O(n)$ steps, Step~$3$ is a standard shuffling procedure (e.g. \cite[Alg.~P (Shuffling)]{knuth1997}) of $n$ elements and takes $O(n)$ steps and Step~$4$ is again performed in $O(n)$ steps.
In our setting we additionally have to take care of the computation time of $(p_0,p_1,\dots,p_K)$ in Step 1. Denoting the expected time to do so by $X$ we obtain that the total expected running time of Algorithm~\ref{algo:size-constrained-gw-tree-devroye} in our setting is
\[
   O\left(\frac{1+ \tau_n + X}{\varphi_n}\right) + O(n).
\]
Lemma~\ref{lem:xi_finite_exp_mom} guarantees that $\Ex{\xi^2}<\infty$, so we immediately obtain from~\cite{MR2888318} that $\tau_n=O(n^{1/2})$ and $\varphi_n^{-1}=O(n^{1/2})$.
This means that it suffices to show that $X=O(n^{1/2})$.
Assumption~\ref{assumption3} yields that there exists a non-negative sequence $(f_i)_{i\in\ndN_0}$ with $f_n\to 0$ as $n\to\infty$ such that the expected time to compute $(p_0,p_1,\dots,p_K)$ is 
\begin{align}
\label{eq:comp_time_pk}
    X
    = \mathbb{E}\bigg[\sum_{0\le i \le K} \mathrm{e}^{f_i\cdot i}\bigg]
    = \sum_{m\ge 1} \Prb{K=m} \sum_{0\le i\le m}\mathrm{e}^{f_i\cdot i}.
\end{align}
%Let $\eps'>0$ be fixed but arbitrary. Then there exists $L\in\ndN$ such that $f_i\le \eps'$ for all $i\ge L$. Hence, if $m\ge L$ there exists a constant $C'>0$
%\[
%    \sum_{0\le i\le m}\mathrm{e}^{f_i\cdot i}
%    = \sum_{0\le i\le L}\mathrm{e}^{f_i\cdot i}
%    + \sum_{L\le i\le m}\mathrm{e}^{f_i\cdot i}
%    \le C' + (m-L) \mathrm{e}^{\eps' m}.
%\]
%If $m<L$ the expression above is simply $\le C'$. 
Since $f_n \to 0$ we obtain that for every $\eps>0$  there is $C>0$ such that for all  $m \in \mathbb{N}$
\[
     \sum_{0\le i\le m}\mathrm{e}^{f_i\cdot i} \le C \cdot \mathrm{e}^{\eps m}.
\]
%Note that since $\eps'$ is arbitrarily small (by choosing a larger $L$) the parameter $\eps$ can be chosen as small as fits our needs as well. 
Thus, continuing with~\eqref{eq:comp_time_pk}
\[
    X
    \le  C\sum_{m\ge 1} \Prb{K=m} \mathrm{e}^{\eps m}
    = C \Ex{\mathrm{e}^{\eps K}}.
\]
By applying Lemma~\ref{lem:xi_finite_exp_mom} we may choose a $t>0$ such that $\Ex{\mathrm{e}^{t\xi}}<\infty$.
We use the Log-Sum-Exp  estimate $\max_{1\le i\le n}x_i \le \log(\mathrm{e}^{\alpha x_1}+\cdots+\mathrm{e}^{\alpha x_n})/\alpha$ for any $(x_1,\dots,x_n)\in\ndN_0^n$ and $\alpha>0$ to obtain 
\[
    \Ex{\mathrm{e}^{\eps K}}
    \le \Ex{\mathrm{e}^{\eps/t\cdot \log( \mathrm{exp}(t\xi_1)+\cdots+\mathrm{exp}(t\xi_n))}}
    = \Ex{( \mathrm{exp}(t\xi_1)+\cdots+\mathrm{exp}(t\xi_n))^{\varepsilon/t}}.
\]
Pick $\eps > 0$ such that $0 < \eps/t < 1/2$.
Then the Jensen inequality entails ($x \mapsto x^{\varepsilon/t}$ is concave)
\begin{align*}
    \Ex{( \mathrm{exp}(t\xi_1)+\cdots+\mathrm{exp}(t\xi_n))^{\varepsilon/t}}
    \le \Ex{ \mathrm{exp}(t\xi_1)+\cdots+\mathrm{exp}(t\xi_n)}^{\varepsilon/t}
    = (n \Ex{\mathrm{e}^{t\xi}})^{\eps/t}
    = o(n^{1/2}).
\end{align*}
\end{proof}
\begin{proof}[Proof of Theorem~\ref{thm:algo_runtime_n}]
Subsequently we go through each step in Algorithm~\ref{algo:R-enriched_tree} and explain how the expected runtime of $O(n)$ is achieved. For Step 1 see Lemma~\ref{lem:devroye_lin_time}. Let us next investigate Step 2.
Recall that $\cA(\rho_{\cA})<t_0<\rho_{\cR}$. 
The probability that $\Gamma\cR(t_0)$ produces a $k$-sized object is  
\[
	\Pr{\lvert\Gamma\cR(t_0)\rvert = k}
	= \frac{r_k t_0^k}{k! \cR(t_0)},
	\qquad k\in\ndN_0.
\]
Hence the time until an object of size $k\in\ndN_0$ is distributed like  a random variable $G(k)$ with geometric distribution and with mean
\[
	g(k)
	= \frac{r_k t_0^k}{k! \cR(t_0)} 
	\sum_{\ell \ge 1} \ell \left( 1 - \frac{r_k t_0^k}{k! \cR(t_0)} \right)^{\ell-1}
	= \frac{k! \cR(t_0)}{r_k t_0^k}.
\] 
Let $W_n$ denote the total waiting time until Step 2 is completed and set $\Xi_n = \sum_{1\le i\le n}\xi_i$. Recall that $\mT_n$ can be equivalently represented by its outdegree sequence as outlined in~\eqref{eq:condition_degree_sequence_GW} so that we deduce
\begin{align}
\label{eq:Ex_W_n}
	\Ex{W_n}
	= \Exb{ \sum_{v\in \mT_n} G({d_{\mT_n}^+(v)})}
	= \sum_{1\le i\le n}\Exb{ G(\xi_i) \Bigm| \Xi_n = n-1, S_t>0\text{ for }1\le t\le n-1 }.
	%= n \Exb{ g(\xi_1) \bigm| \Xi_n = n-1}.
\end{align}
Next we use the fact, sometimes referred to as the \emph{cycle lemma}, see for example~\cite{MR1034142}, that for any sequence $(d_1,\dots,d_n)$ of non-negative integers such that $\sum_{1\le i\le n}d_i = n-1$ there exists a unique $1\le \ell\le n$ such that the shifted sequence $(\tilde{d}_1,\dots,\tilde{d}_n)=(d_\ell,\dots,d_n,d_1,\dots,d_{\ell-1})$ fulfils that $1+\sum_{1\le i\le t}(\tilde{d}_i-1) >0$ for $1\le t\le n-1$. Since this shift is only one of $n$ possible shifts and we are dealing with iid random variables we obtain
\[
    \Prb{\Xi_n = n-1, S_t>0\text{ for }1\le t\le n-1}
    = \frac{1}{n} \Prb{\Xi_n = n-1},
\]
a well-known relation. 
Further, as $\xi_1,\xi_2,\dots$ are iid the random variable $G(\xi_i)$ is not altered by rotating $(\xi_1,\dots,\xi_n)$ so that we also get for any $k\in\ndN$
\[
    \Prb{G(\xi_i)=k,\Xi_n = n-1, S_t>0\text{ for }1\le t\le n-1}
    = \frac{1}{n}  \Prb{G(\xi_1)=k,\Xi_n = n-1},
    ~~ 1\le i\le n.
\]
Combining the latter two equations into \eqref{eq:Ex_W_n}  yields
\begin{align}
    \label{eq:W_n}
    \Ex{W_n}
    = n\Exb{G(\xi_1) \Bigm| \Xi_n = n-1}.
\end{align}
Next we compute
\begin{align*}
    \Exb{G(\xi_1) \Bigm| \Xi_n = n-1}
    &= \sum_{k,\ell\ge 0} \ell \Prb{G(k)=\ell}
    \Prb{\xi_1=k\Bigm| \Xi_n = n-1} \\
    &= \sum_{k\ge 0}g(k)\Prb{\xi_1=k\Bigm| \Xi_n = n-1}.
\end{align*}
With this at hand we obtain with Lemma~\ref{lem:local_central_limit_theorem} (which we are allowed to apply due to Lemma~\ref{lem:xi_finite_exp_mom}) that 
\begin{align}
	\label{eq:GG}
	 \Exb{G(\xi_1) \Bigm| \Xi_n = n-1}
	&= \sum_{k\ge 0} g(k) \Pr{\xi=k} \frac{\Pr{\Xi_{n-1} = n-k-1}}{\Pr{\Xi_n = n-1}}  \\
	%= O(1) \sum_{k\ge 1} g(k)  \Pr{\xi=k} \\
	&= O\left( \Ex{g(\xi)} \right). \nonumber
\end{align}
Since $t_0>\cA(\rho_{\cA})$, 
\begin{align}
	 \Ex{g(\xi)} =  \sum_{k\ge 0} \frac{\cR(t_0)}{\cA(\rho_{\cA})} \left(\frac{\cA(\rho_{\cA})}{t_0}\right)^k  < \infty.
\end{align}
Looking back at~\eqref{eq:W_n} this immediately gives us $\Ex{W_n} = O(n)$, as desired. Finally, generating a uniform permutation of $[n]$ in Step 3 takes time $O(n)$, see for example \cite[Alg.~P (Shuffling)]{knuth1997}.
\end{proof}

%###################################### Applications #######################################
\section{Main Applications}

We study several models of random trees, graphs and permutations. For each we explain in detail how they fit into the general framework of $\cR$-enriched trees and devise linear time exact size sampling algorithms by verifying the conditions of our main theorem.

\label{sec:applications}
% Due to space limitations we only present an application of our results to the class of series-parallel graphs (which is a special case of subcritical classes of connected graphs). The other examples will be treated in detail in a future journal version of this work. Nonetheless, it is far from obvious how to assemble from Algorithm~\ref{algo:R-enriched_tree} a sampler for Bienaym\'e--Galton--Watson trees conditioned on the number of nodes whose outdegree lies in a given set. Thus, for completeness, we devoted Section~\ref{sec:gw_leaves} in the Appendix to this problem. 

%\ben{Evtl. kurz herausarbeiten dass diese Anwendungen auch Teil der "main contribution" des papers sind.}

\subsection{Subcritical classes of connected graphs}
\label{ssec:sccg}

Let us first recall a few basic notions from graph theory. A subgraph of a graph $G$ is called a \emph{block} of $G$ if is is maximal such that it is either (isomorphic) to an edge or if it is 2-connected otherwise. We call a class $\cG$ \emph{block-stable} if it contains the graph that is isomorphic to a single edge and that has the property that a graph belongs to $\cG$ if and only if all of its blocks belong to $\cG$. Block stable classes are ubiquitous, and include, for example, classes that are specified by excluding a finite list of minors. 

It is well-known, see for example~\cite{MR2465772}, that any block stable class $\cC$ of connected graphs satisfies the decomposition
\begin{align}
	\label{eq:blockdecomp}
    \cC^\bullet \simeq \cX\cdot (\Set\circ\cB'\circ\cC^\bullet),
\end{align}
where $\cB$ denotes the class of 2-connected graphs in $\cC$ (together, possibly, with the graph that is isomorphic to an edge) and $\cX$ contains a single vertex. The combinatorial constructions used in that formula are defined in the preliminaries, see Section~\ref{ssec:classes}. This decomposition of connected graphs  is a consequence from the decomposition into blocks that enables us to describe a graph in terms of a tree and associated subgraphs of higher connectivity.
Roughly speaking, Equation~\eqref{eq:blockdecomp} follows from the fact that the root vertex of a connected rooted graph is incident to an unordered collection of blocks, that is, maximal $2$-connected subgraphs. The entire graph may be described by this collection, with arbitrary rooted connected graphs glued to each of its non-root vertices.

What is important here is that block stable classes are isomorphic to enriched trees $\cA_\cR$ with $\cR = \Set \circ \cB'$; this observation is not new and was also exploited  elsewhere~\cite{MR3551197,StEJC2018,MR4132643}. The correspondence is illustrated in Figure~\ref{fi:block}. Roughly speaking, we ``unroll'' the decomposition in~\eqref{eq:blockdecomp}.
That is, given a rooted connected graph, we form the collection of blocks incident to the root vertex. This collection will become the decoration of the root of the associated tree.
The non-root vertices of this collection correspond to the children of the root of the associated tree.
As the entire graph consists of this collection with arbitrary rooted graphs glued to each of its non-root vertices, this transformation is then recursively applied to each of these graphs, grafting the resulting enriched trees to the children of the root in the enriched tree.

%Moreover, if the graphs are represented by adjacency lists, it is rather straightforward to implement the bijection between the class of enriched trees $\cA_\cR$ and the associated class $\cC$ by traversing the tree in some order (for example breadth-first search) and gluing the respective blocks together as just described and depicted in Fig.~\ref{fi:block}. The cost of doing so is bounded by the a multiple of the total sum of the degrees of all vertices and is this linear in the size of the created graph. \marginpar{\tiny NEU, OK so?}

\begin{figure}[t]
	\centering
	\begin{minipage}{1.0\textwidth}
		\centering
		\includegraphics[width=0.97\textwidth]{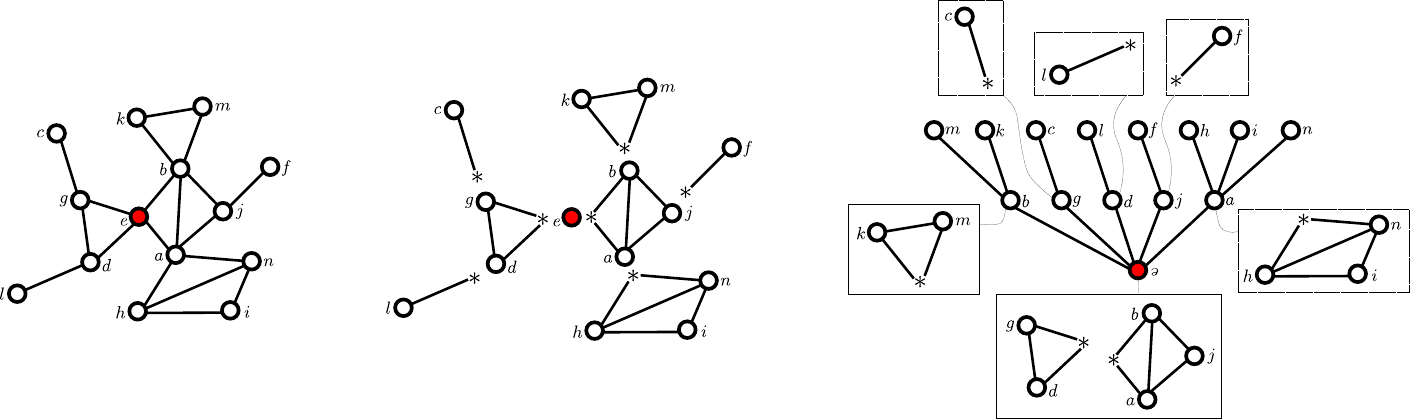}
		\caption{A rooted connected graph, its block decomposition, and the corresponding $\cR$-enriched tree.}
		\label{fi:block}
	\end{minipage}
\end{figure}

From the description of block stable classes, see also~\eqref{eq:relRAR},  we immediately obtain the relation 
\[
    \cC'(x) = \exp(\cB'(x \cC'(x)))
\]
for the corresponding exponential generating functions. We will study here the particular case in which the composition of the generating functions is \emph{subcritical}, meaning that the largest value that $x \cC'(x)$ can attain, where $x$ is at most the radius of convergence of $\cC(x)$, lies strictly within the disc of convergence of~$\cB'(x)$.
\begin{definition}
	\label{def:subcritgraph}
Let $\cC$ be a block stable class of connected graphs and $\cB$ the corresponding class of 2-connected graphs together with the graph that is isomorphic to an edge. Let $\rho_\cC$ and $\rho_\cB$ denote the radii of convergence of $\cC(x)$ and $\cB(x)$. We say that $\cC$ is \emph{subcritical} if $\rho_\cB > \rho_\cC \cC'(\rho_\cC)$.
\end{definition}
Subcritical graph classes have been studied from various viewpoints and include many important classes like trees, outerplanar and series-parallel graphs, but exclude also others, like planar graphs. From an \emph{analytical} viewpoint, in the subcritical case the behavior of $\cC(x)$ near its singular points is not dictated by the behavior of $\cB(x)$, but it is rather a consequence of the composition $\cB'(x\cC'(x))$; this is in stark contrast to critical compositions (where necessarily $\rho_\cB = \rho_\cC \cC'(\rho_\cC)$), where there is an explicit interplay. From a \emph{combinatorial} viewpoint subcritical classes are very much tree-like, in the sense that the blocks are typically small, the largest one having at most logarithmic size in the size of the graph. This makes it possible to study subcritical classes rather abstractly without explicitly fixing $\cB$, and by now there are many results that address the local as well the global structure of 'typical' members of such classes~\cite{MR3184197,MR2534261,MR2873207,MR2675698}. Here we extend this list of results by providing linear-time exact-size samplers for many important classes including cactus, outerplanar and series-parallel graphs; many other classes for which there is a decomposition of the 2-connected graphs can be treated analogously.

In order to apply our main result we begin with the following fact, which is a simple consequence of the definition of subcritical classes.
\begin{fact}
Let $\cC$ be subcritical. Then $\cC$ has the properties~\ref{assumption2} and~\ref{assumption1} in the definition of a tame class of $\cR$-enriched trees for $\cR= \Set \circ \cB'$. 
\end{fact} 

So, what remains to be done in order to obtain an expected linear-time exact-size sampler for $\cR$-enriched trees by applying Theorem~\ref{thm:algo_runtime_n} is to verify the properties~\ref{assumption3} and~\ref{assumption4}  in the definition of tame enriched trees for the classes at hand. In particular, we have to show that $|\cR_k| = |\Set \circ \cB'_k|$ can be computed in time $e^{o(k)}$ and that we have a Boltzmann sampler for $\cR = \Set \circ \cB'$. \emph{Assuming} for the moment that this can be done (we will verify this condition for cacti graphs, outerplanar graphs, and series-parallel graphs in the following subsections), we  may apply Theorem~\ref{thm:algo_runtime_n} that utilizes Algorithm~\ref{algo:R-enriched_tree} with the a random variable $\xi$ that has distribution, see also~\eqref{eq:def_offspring_distribution_R-enriched},
\begin{align}
    \mathbb{P}(\xi = k)
    = \frac{|\cR_k| \rho_\cC \cC^\bullet(\rho_{\cC})^k }{\cC(\rho_\cC) k!},
    \quad
    \text{with pgf}
    \quad
    \mathbb{E}[z^\xi] = \frac{\rho_\cC\cR(\cC^\bullet(\rho_\cC) z)}{\cC^\bullet(\rho_{\cC})},
\end{align}
to obtain a linear-time exact-size sampler for $\cR$-enriched trees.  
In order to obtain a linear-time exact-size sampler for $\cC$ we need to transform this enriched into a rooted graph and drop the root-vertex:

\begin{algo} (Uniform $n$ vertex graphs from a subcritical graph class $\cC$)
	\label{algo:unigraph}
	Let $\cC$ be subcritical and assume that properties~\ref{assumption3} and~\ref{assumption4}  in the definition of tame enriched trees are met. Then the following generates a uniform $n$-vertex graph from $\cC$ in expected time $O(n)$.
	\label{algo:graphs}
	\begin{enumerate}
		\item Generate a uniform $n$-vertex $\cR= \Set \circ \cB'$ enriched tree using Algorithm~\ref{algo:R-enriched_tree}.
		\item Transform the enriched tree into a rooted graph from $\cC^\bullet$.
		\item Create an unrooted graph by forgetting which was the root vertex.
	\end{enumerate}
\end{algo}

First, the algorithm is correct, since a uniform enriched tree corresponds to a uniform rooted connected graph. Any $n$-vertex labelled graph corresponds to $n$ rooted versions, hence forgetting the root vertex yields the uniform distribution on $n$-vertex labelled unrooted graphs.
	
Second,  Theorem~\ref{thm:algo_runtime_n} guarantees that generating a uniform $n$-vertex enriched tree takes expected time $O(n)$. Transforming it into a graph takes expected time $O(n)$ by the following lemma, hence Algorithm~\ref{algo:unigraph} runs in expected linear time.

\begin{lemma}
	A uniformly selected $n$-vertex $\cR=\Set \circ \cB'$ enriched tree may be transformed into a rooted connected graph from $\cC^\bullet$ in an expected linear number of steps.
\end{lemma}
\begin{proof}
	%Moreover, if the graphs are represented by adjacency lists, 
	Let $(T, (R_v)_{v \in T})$ denote an $n$-vertex $\cR$-enriched tree for $\cR= \Set \circ \cB'$. We may construct the associated rooted graph by traversing the tree $T$ in some order (for example breadth-first-search) and gluing the respective blocks together as depicted in Figure~\ref{fi:block}. 
	
	At any point in this traversal, when we arrive at a vertex $v \in T$ we glue the blocks from $R_v$ to the graph constructed so far. If the graphs are represented by adjacency lists, then the time required for this step is bounded by a constant multiple of the number of edges in the blocks from $R_v$. As the number of non-$*$-vertices of $R_v$ agrees with the number $d_T^+(v)$ of children of $v$ in $T$, the number of edges in $R_v$ is at most $O(d_T^+(v)^2)$.
	
	Hence the time required for transforming $(T, (R_v)_{v \in T})$ into a graph is bounded by a constant multiple of $\sum_{v \in T} d_T^+(v)^2$. Hence it follows from Lemma~\ref{le:disenrichedtree} that for a uniformly selected $n$-vertex $\mA_n (\mT_n,(\mR_v)_{v\in \mT_n})$ the expected required time is  bounded by a constant multiple of
	\[
	\Exb{ \sum_{v \in \mT_n} d^+_{\mT_n}(v)^2}.
	\]
	By identical arguments as for Equations~\eqref{eq:Ex_W_n},~\eqref{eq:W_n}, and~\eqref{eq:GG}, it follows that this bound belongs to $O(n)$ and the proof is complete.
\end{proof}

The following subsections are devoted to verifying~\ref{assumption3} and~\ref{assumption4} for the classes of cactus, outerplanar and series-parallel graphs, which are all subcritical, see for example~\cite{MR3551197}. Further minor-closed classes (such as the class of graphs that contain no cycle of length at least five) that fall into the present setting were described in~\cite{zbMATH06569061}, but we omit the details.

\subsubsection*{Cactus Graphs}

A cactus graph is a graph in which each edge is contained in at most one cycle, that is, the class of such graphs is the block-stable class in which each every block is either an edge or a cycle. So, for this class we have
\[
    \cB = e + \Cyc(\cX),
    \quad
    \text{and}
    \quad
    \cR = \Set \circ \cB',
\]
where $e$ is class of graphs that consists only of a single edge of size two. We immediately obtain that, see the relations for generating functions in Section~\ref{ssec:classes} or directly in~\cite[Sec.~8.4]{MR3551197},
\[
    \cB'(x) = x + \frac{x^2}{2(1-x)},
    \quad
    \cR(x) = (\Set \circ \cB')(x) = e^{\cB'(x)}.
\]
In order to verify Condition~\ref{assumption3} we have to show how to compute $|\cR_k|$ in time $e^{o(k)}$. In this case, the counting sequence for $\cB'$ is explicit: we have $|\cB'_1| = 1$ and $|\cB'_k| = k!/2$ for $k \ge 2$. Then, computing the $k$-th coefficient of $e^{\cB'(x)}$ can readily be done in polynomial (actual, at most cubic) time by using for example the recursive method~\cite{MR510047}.

It remains to verify Condition~\ref{assumption4}.
From the general principles for the construction of Boltzmann samplers~\cite{MR2095975} we infer that a Boltzmann distributed object from $\cR$ contains a Poisson number of independent $\cB'$ components, each of which is also Boltzmann distributed. Moreover, a Boltzmann distributed object from $\cB'$ is with some constant probability an edge, and with the remaining probability a Boltzmann distributed cycle. To be completely explicit in this example, the Boltzmann sampler for $\cR$ is given by the following algorithm.

\begin{algo}
\label{algo:genericSetB}
Boltzmann sampler $\Gamma \cR(t)$ for the class of cactus graphs, where $0 \le t < \rho_\cR$.
	\begin{enumerate}
		\item Let $\ell$ be Poisson distributed with parameter $\cB'(t)$.
        \item For each $1\le k \le \ell$ let independently $b'_i$ be a random graph from the Boltzmann distribution from $\cB'$ with parameter $t$.
%		\item For $1\le k \le \ell$ let independently $e_i$ be equal to one with probability $t / \cB'(t)$ and zero otherwise.
%		\item For each $1\le k \le \ell$ let $b'_k \in \cB'$ be an edge if $e_k = 1$, and otherwise a cycle with $2 + d_k$ vertices, where $d_k$ is an independent geometrically distributed random variable with parameter $t$. In each case, a uniformly random vertex of the created object is marked.
		\item Distribute uniformly at random labels from $[\sum_{1 \le k \le \ell} |b_k'|]$ to $b'_1, \dots, b'_\ell$ and return the collection of relabeled $b'_1, \dots, b'_\ell$.
	\end{enumerate}
\end{algo}
Note that this algorithm is actually a \emph{generic} algorithm for sampling from the Boltzmann distribution for a class $\Set \circ \cB'$, provided that we have a corresponding sampler for $\cB'$. We will (re-)use this algorithm in the following examples -- outerplanar and series-parallel graphs -- as well. 
It remains to specify the Boltzmann sampler for $\cB'$.
\begin{algo}
\label{algo:cactusbi}
A Boltzmann sampler $\Gamma \cB'(t)$ for the class of connected cactus graphs without a cut vertex, where $0 \le t < \rho_\cB = \rho_\cR = 1$.
	\begin{enumerate}
%		\item For $1\le k \le \ell$ let independently $e_i$ be equal to one with probability $t / \cB'(t)$ and zero otherwise.
%		\item For each $1\le k \le \ell$ let $b'_k \in \cB'$ be an edge if $e_k = 1$, and otherwise a cycle with $2 + d_k$ vertices, where $d_k$ is an independent geometrically distributed random variable with parameter $t$. In each case, a uniformly random vertex of the created object is marked.
		\item Let $h$ be equal to one with probability $t / \cB'(t)$ and zero otherwise.
		\item If $h=1$ create an edge and otherwise a cycle with $2 + d$ vertices, where $d$ is geometrically distributed with parameter $t$.
		\item Replace in the created graph the largest label with $\star$ and return.
	\end{enumerate}
\end{algo}

\subsubsection*{Outerplanar Graphs}

An outerplanar graph is a planar graph that can be embedded in such a way that every vertex lies on the boundary of the outer face. In this case, the blocks essentially correspond to edges and dissections of polygons, see for example~\cite[Sec.~8.5]{MR3551197} for the (well-known) following statement.
\begin{lemma}
\label{lem:decompouterplanarbi}
Let $\cB$ be the class of all connected outerplanar graphs not containing a cut-vertex. Then there is a bijection
\[
    \cB' + \cB' \simeq \cX + \cD,
\]
where the class $\cD$ of dissections satisfies $\cD = \cX + \Seq_{\ge 2} \circ \cD$.
\end{lemma}
See Figure~\ref{fi:decomp1} for the specification of the class of dissections and the corresponding bijection.
Lemma~\ref{lem:decompouterplanarbi} enables us to compute via the recursive method~\cite{MR510047} $[x^k]\cB'(x)$ and $[x^k]e^{\cB'(x)}$ in time $e^{o(k)}$ with plenty of room to spare; this verifies Condition~\ref{assumption3}.
In order to verify~\ref{assumption4} we use, as already announced, Algorithm~\ref{algo:genericSetB} as the Boltzmann sampler for $\cR = \Set\circ \cB'$, where we additionally need to specify the sampler for $\cB'$.
This, in turn, following the general Boltzmann sampling principles~\cite{MR2095975}, can be realized by making a two-way choice between $\cX$ (with probability $t/2\cB'(t)$) and $\cD$ (with the remaining probability), in complete analogy to what we did in Algorithm~\ref{algo:cactusbi}.
Finally, the sampler for $\cD$ can be immediately obtained from the specification: we make (again) a two-way choice between $\cX$ (this time with probability $t/\cD(t)$) and $\Seq_{\ge2} \circ \cD$, where in the latter the number of $\cD$ components follows a geometric distribution, that is conditioned to be $\ge 2$, with parameter $\cD(t)$.

\begin{figure}[t]
	\centering
	\begin{minipage}{0.7\textwidth}
  		\centering
  		\includegraphics[width=1.0\textwidth]{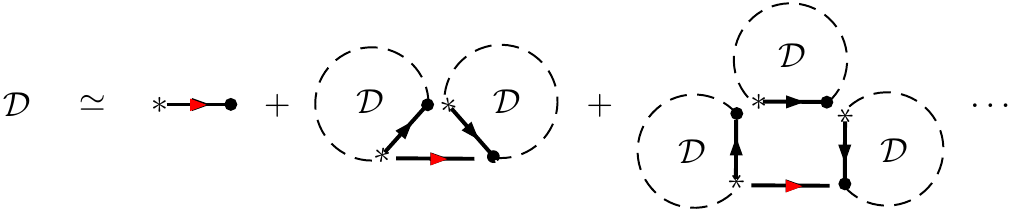}
  		\caption{Recursive specification of the class $\cD$.}
  		\label{fi:decomp1}
	\end{minipage}
\end{figure}

\subsubsection*{Series-Parallel Graphs}
In the case of series-parallel graphs we will use the following property of the associated class $\cB'$ (connected series-parallel graphs with no cut-vertex).
\begin{lemma}
\label{lem:decomp2connSP}
Let $\cX$ be the class of graphs consisting of a single graph that is an isolated vertex, $\cX_2$ the class consisting of a single graph that contains exactly two isolated vertices, and $e$ the class consisting of a single graph that is an isolated edge, where the size of it is defined to be zero. Then 
the class $\cB'$ has the decomposition
\[
	\cB' + \cB'_{(rm)}
	\simeq
	e\cdot \cX + \cB'_{(r)} + \cB'_{(m)},
\]
where $\cB'_{(rm)} \subseteq \cB'$ and
\[
	\cB'_{(r)}
	    \simeq
	\cX_2 \cdot (e + \cP)^2 \cdot \cD,
	\enspace
	\cB'_{(m)} \simeq \cX \cdot \big(e \cdot \Set_{\ge 2}(\cS) + \Set_{\ge 3}(\cS)\big),
	\enspace
	\cB'_{(rm)} \simeq \cX \cdot \cP \cdot \cS,
\]
and
\[
	\cD \simeq e + \cS + \cP,
	\cS = (e + \cP) \cdot \cX \cdot \cD,
	\cP = e \cdot \Set_{\ge_1}(\cS) + \Set_{\ge_2}(\cS).
\]
\end{lemma}
This lemma is taken from~\cite[Sec.~6.1]{MR2534261}, where also the simple bijections -- that can be implemented in linear time -- behind the given  isomorphisms are depicted. We will not repeat this here, since it would be a mere reconstruction of what is done in~\cite{MR2534261}.
Let us just mention that in Lemma~\ref{lem:decomp2connSP} the classes $\cD,\cS$ and $\cP$ correspond to the well-known classes of general, series and parallel networks, which are, roughly speaking, 2-connected series-parallel graphs with a removed edge whose endpoints are distinguished and do not contribute to the size. Their decomposition has been known since the 1980's.
Moreover, the decomposition of $\cB'$ follows from the general decomposition scheme of connected graphs in their 3-connected components; for details we refer to our primary source~\cite{MR2534261} and the paper~\cite{MR2465772}.

With Lemma~\ref{lem:decomp2connSP} at hand we proceed as usual in this section. We are again in a position to compute via the recursive method~\cite{MR510047}  $[x^k]\cB'(x)$ and $[x^k]e^{\cB'(x)}$ in time $e^{o(k)}$ with room to spare; this verifies Condition~\ref{assumption3}. In order to verify~\ref{assumption4} we again use Algorithm~\ref{algo:genericSetB} as the Boltzmann sampler for $\cR = \Set\circ \cB'$, where the last remaining step is to specify the sampler for $\cB'$. Using the decomposition in Lemma~\ref{lem:decomp2connSP} we readily obtain a Boltzmann sampler for $\cB' + \cB'_{(rm)}$ from the general principles for the construction of Boltzmann samplers from~\cite{MR2095975}. In addition to that, since $\cB'_{(rm)} \subseteq \cB'$, we obtain a sampler for $\cB'$ by rejection: if the sampled graph is in $\cB'_{(rm)}$, we reject it with probability $1/2$ and repeat the experiment. 

\subsection{Bienaym\'e--Galton--Watson trees conditioned on the number of vertices with given degrees.}
\label{sec:gw_leaves}

Throughout this section we fix a proper subset $\Omega \subsetneqq \ndN_0$ satisfying $0 \in \Omega$.  See Remark~\ref{rem:conditions} below for comments on this assumption. We let $\Omega^c := \ndN_0 \setminus \Omega$ denote its complement. Let $\zeta$ be a random non-negative integer satisfying 
\begin{align}
\label{eq:cond_zeta}
	\Pr{\zeta= 0} >0, \quad \Pr{\zeta \ge 2} >0, \quad \text{and} \quad \Pr{\zeta \in \Omega} > 0.
\end{align}
Our aim in this section is to develop an expected linear-time sampler for a tree $\mA_n^\Omega$ that is distributed like  a $\zeta$-Galton--Watson $\mA$ tree conditioned on having $n$ vertices with outdegree in $\Omega$. Formally, letting $L_\Omega(\cdot)$ denote the number of vertices with outdegree in $\Omega$, we set $\mA_n^\Omega= (\mA \mid L_\Omega(\mA) = n)$.  Of course, we only consider integers $n$ for which this is well defined, that is, where $\Pr{L_\Omega(\mA) =n} > 0$.  Furthermore, in order to obtain a procedure that samples in linear time, we assume that
\begin{align}
	\label{eq:finexpzeta}
	\Ex{\zeta} = 1 \qquad \text{and} \qquad \Ex{(1+\epsilon)^\zeta} < \infty \quad \text{for some $\epsilon>0$}.
\end{align}
We also assume that the weight $\Pr{\zeta=k}$ may be computed in time $\exp(o(k))$ for each $k \ge 0$, and that the probability $\Pr{\zeta \in \Omega}$ is also given.

A result by Kortchemski~\cite[Thm.~8.1]{MR2946438} asserts that for some constant $c = c(\Omega) >0$
\begin{align}
	\label{eq:rhoaaa}
	\Pr{L_\Omega(\mA) = n} \sim c n^{-3/2}.
\end{align}
Hence, we may use Boltzmann sampling as described in the introduction (rejection and truncation) to obtain a polynomial time exact-size sampler for $\mA_n^\Omega$. This performance is not optimal, hence our motivation for describing a generator that accomplishes this in linear time.

The procedure we are going to describe is based on the fact  that rooted trees satisfying  $L_\Omega(\cdot) = n$ correspond bijectively to $n$-vertex $\cR$-enriched trees for a specific class $\cR$, see  \cite{MR1284403,MR3335013}.  Since  we may generate $\cR$-enriched trees in expected linear time via Algorithm~\ref{algo:R-enriched_tree}, and since the transformation to a plane tree with $L_\Omega(\cdot) =n$ also takes expected linear time, we will arrive at generator for $\mA_n^\Omega$ that runs in expected time~$O(n)$.
To be fully precise, we will use a straight-forward extension of Algorithm~\ref{algo:R-enriched_tree} to weighted species, because $\mA_n^\Omega$ is not (necessarily) uniform among all plane trees~$A$ with $L_\Omega(A) = n$. To wit, for a tree $A$ 
\begin{align}
	\Pr{\mA_n^\Omega= A} = \frac{\Pr{\mA = A}}{\Pr{L_\Omega(\mA) = n}}  =  \frac{1}{\Pr{L_\Omega(\mA) = n}} \prod_{v \in A} \Pr{\zeta = d^+_A(v)}.
\end{align}
Consequently, the random $n$-vertex $\cR$-enriched tree corresponding to $\mA_n^\Omega$ is not (necessarily) uniform. %Instead, it assumes an enriched tree  corresponding to $A$ with the same probability that $\mA_n^\Omega$ assumes $A$.
Furthermore, again to be fully precise, Algorithm~\ref{algo:R-enriched_tree} is formulated for labelled structures. The plane trees we generate here are asymmetric unlabelled structures. That is, it is irrelevant whether we consider them as labelled or unlabelled, since they have no non-trivial symmetries. Thus, we may safely ignore labels in this section.
%Algorithm~\ref{algo:R-enriched_tree} still applies. Kosta: ich verstehe nicht worauf der applies

Let us start with the description of the weighted species $\cR$ in question. 
For each integer $k \ge 0$ we let $\cR_k$ denote the collection of all tuples $R = (y, x_1, \ldots, x_\ell)$ satisfying $\ell \ge 0$, $y \in \Omega$,  $x_1, \ldots, x_\ell \in \Omega^c -1$, and $y + \sum_{i=1}^\ell x_i = k$. To each such tuple $R$ we assign a weight $\gamma(R)$  by
\begin{align}
	\gamma(R) = \Pr{\zeta=y} \prod_{i=1}^\ell \Pr{\zeta= x_i +1}.
\end{align}
It was shown in \cite{MR1284403,MR3335013} in a more general context that an ordered rooted tree $A$
%with $L_\Omega(A) = n$
corresponds bijectively to a pair $(T,\beta)$ of an ordered rooted tree $T$ with $n$ vertices, and a map~$\beta$ that assigns to each inner vertex $v \in T$ a structure $\beta(v) \in \cR_{d^+_T(v)}$. Here $A$ is constructed from~$T$ by a blow-up procedure that replaces a vertex~$v \in T$ and the edges to its children by a tree constructed from $\beta(v)$ as illustrated in Figure~\ref{fi:blowup}. 

\begin{figure}[h]
	\centering
	\begin{minipage}{1.0\textwidth}
		\centering
		\includegraphics[width=0.50\textwidth]{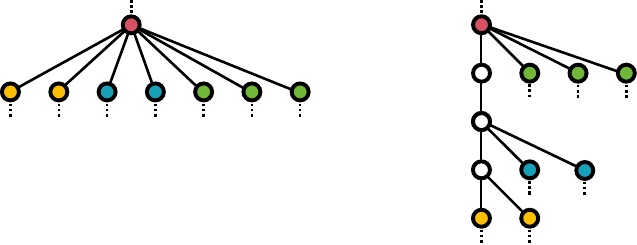}
		\caption{Blow-up procedure of a vertex $v$ (red) having $5$ children and decoration $\beta(v) = (2,2,0,3)$. }
		\label{fi:blowup}
	\end{minipage}
\end{figure}

This  correspondence is weight-preserving in the sense that the weight $\Pr{\mA= A}$ of the tree $A$ equals the weight $\prod_{v \in T} \gamma(\beta(v))$ of the decorated tree $(T, \beta)$.
Furthermore,  the random decorated tree $(\mT_n, \beta_n)$ corresponding to the random tree $\mA_n^\Omega$ has the property that $\mT_n$ is distributed like a $\xi$-Galton--Watson tree conditioned on having $n$ vertices, where the random integer $\xi \ge 0$ has probability generating function
\begin{align}
	\label{eq:exprxi}
	\Ex{z^\xi} = \left( \sum_{k \in \Omega} \Pr{\zeta = k} z^k \right) \left(1 - \sum_{k \in \Omega^c} \Pr{\zeta=k} z^{k-1}  \right)^{-1}.
\end{align}
Conditional on $\mT_n$, each decoration $\beta_n(v)$ gets drawn from  $\cR_{d^+_{\mT_n}(v)}$ with probability proportional to its $\gamma$-weight, independently from the rest. Setting $|\cR_k|_\gamma= \sum_{R \in \cR_k} \gamma(R)$ for each $k \ge 0$, the ordinary generating series  of the species $\cR$ is given by \begin{align}
	\cR(z) := \sum_{k \ge 0} |\cR_k|_\gamma z^k = \Ex{z^\xi}.
\end{align}
See~\cite{MR1284403,MR3335013,MR4132643} for detailed justifications.

Our strategy for generating $\mA_n^\Omega$ in expected time $O(n)$ is to generate the $\cR$-enriched plane tree $(\mT_n, \beta_n)$ with Algorithm~\ref{algo:R-enriched_tree} and apply the blow-up procedure. In order to verify that this works we have to do two things. First, we have to check that the
%weighted analogues of the
conditions of Algorithm~\ref{algo:R-enriched_tree} are met. Second, we have to check that the expected time for applying the blow-up procedure to $(\mT_n, \beta_n)$ is~$O(n)$. Note that there is a subtle difficulty in the second step, because the number of vertices of $\mA_n^\Omega$ may be much larger than $n$ and the blow-up procedure may take very long.
%Hence we really have to bound the expected duration.

\begin{comment}
This correspondence is weight-preserving in the sense that if we define the weight of $A$ by
\begin{align}
	\omega(A) =  \prod_{v \in A} \Pr{\zeta = d^+_A(v)},
\end{align}
then it holds that
\begin{align}
	\omega(A) = \prod_{v \in T} \gamma(\beta(R)).
\end{align}
\end{comment}

Let us start by verifying the conditions of Algorithm~\ref{algo:R-enriched_tree}. 
Condition~\eqref{eq:finexpzeta} and the definition of the probability generating function of $\xi$ in~\eqref{eq:exprxi} entail that $\cR(z) = \Ex{z^\xi}$ has radius of convergence
%\begin{align}
	$\rho_\cR > 1$. (Specifically, $\rho_\cR$ is the supremum of the collection of all $x>0$ for which $\Ex{x^{\zeta} \one_{\zeta \in \Omega}}< \infty$ and $\Ex{x^{\zeta-1} \one_{\zeta \in \Omega^c}}< 1$.)
%\end{align}
For any parameter $t>0$ with $\cR(t)<\infty$ we  define a Boltzmann sampler $\Gamma \cR (t)$  with distribution
\begin{align}
	\Pr{\Gamma \cR(t) = R} = \frac{\gamma(R) t^{k(R)}}{\cR(t)}, \qquad R \in \bigsqcup_{k \ge 0} \cR_k
\end{align}
for $k(R) \ge 0$ the unique integer with $R \in \cR_{k(R)}$.

\begin{algo}{A Boltzmann sampler $\Gamma \cR(t)$:}
	\begin{enumerate}
		\item Generate a random integer $y$ with probability generating function
		$$
		    \Ex{(zt)^\zeta\one_{\zeta \in \Omega}} / \Ex{t^\zeta \one_{\zeta \in \Omega}}.
        $$
		%$\sum_{k \in \Omega} \Pr{\zeta=k} (zt)^k / \sum_{k \in \Omega} \Pr{\zeta=k} t^k$.
		\item Generate a random integer $\ell$ with geometric distribution with parameter $\Ex{t^{\zeta-1} \one_{\zeta \in \Omega^c}}$.
		%That is, its probability generating function equals $(1 - \Pr{\zeta \in \Omega^c}) / (1 - z \Pr{\zeta \in \Omega^c})$.
		\item For each $1 \le i \le \ell$ generate a random integer $x_i$ with probability generating function given by $\Ex{t^{\zeta-1} \one_{\zeta \in \Omega^c}}^{-1} \sum_{k \in \Omega^c} \Pr{\zeta=k} (zt)^{k-1}$. 
		\item Return $(y, x_1, \ldots, x_\ell)$.
	\end{enumerate}
\end{algo}

For any integer $k \ge 0$ (satisfying $|\cR_k|_\gamma>0$) we may condition $\Gamma \cR(t)$ on returning an element from~$\cR_k$. The element generated in this way is drawn with probability proportional to its $\gamma$-weight from $\cR_k$. Since all coordinates of a tuple from $\cR_k$ are at most $k$, we may fix some $K \ge k$ and work with a truncated version $\Gamma_{\le K} \cR(t)$ instead. That is,  $\Gamma_{\le K} \cR(t)$ uses truncated versions $(y \mid y \le K)$ and $(x_i \mid x_i \le K)$ instead, and conditioning $\Gamma_{\le K} \cR(t)$ on producing an element from $\cR_k$ also  yields a random element that gets drawn with probability proportional to its $\gamma$-weight. Furthermore, constructing $\Gamma_{\le K} \cR(t)$ only requires knowledge of $\Pr{\zeta \in \Omega}$ and the probabilities $\Pr{\zeta=i}$ for $0 \le i \le K$. We assumed that $\Pr{\zeta \in \Omega}$ is given and that $\Pr{\zeta=i}$ may be computed in $e^{o(i)}$ steps. Hence constructing $\Gamma_{\le K} \cR(t)$ requires $e^{o(K)}$ preprocessing time.
Running a single instance of $\Gamma_{\le K} \cR(t)$  only requires constant time in expectation. Moreover, $\Pr{\Gamma_{\le K} \cR(t) \in \cR_k} \ge \Pr{\Gamma \cR(t) \in \cR_k}$, hence generating a random element from $\cR_k$ using $\Gamma_{\le K} \cR(t)$ is at least as fast as using $\Gamma \cR(t)$.

%Note that we assumed that $\Pr{\zeta=k}$ may be calculated in time $e^{o(k)}$. Consequently, $\Pr{\xi=k}$ may be calculated in time $e^{o(k)}$ as well. 

We may now state our final algorithm for sampling $\mA_n^\Omega$.

\begin{algo} A generator for $\mA_n^\Omega$ that runs in expected time $O(n)$.
	\label{algo:anomega}
	\begin{enumerate}
		\item Use Algorithm~\ref{algo:size-constrained-gw-tree-devroye} to sample a  Bienaym\'e--Galton--Watson tree $\mT_n$ with offspring distribution $\xi$ conditioned on having $n$ vertices.
		\item Let $K$ denote the maximal outdegree of $\mT_n$. For a fixed $1 < t_0 <\rho_{\cR}$ repeatedly call for each $v\in \mT_n$ the sampler $\Gamma_{\le K} \cR(t_0)$ until it produces an object $\beta_n(v)$ from $\cR_{d^+_{\mT_n}(v)}$.
		\item Perform the blow-up procedure illustrated in Figure~\ref{fi:blowup} on $(\mT_n, \beta_n)$ to create $\mA_n^\Omega$.
	\end{enumerate}
\end{algo}
Here's a justification why Algorithm~\ref{algo:anomega} runs in expected time $O(n)$.
\begin{proof}
We first show that Step (1) can be implemented in expected time $O(n)$.
The expression of $\Ex{z^\xi}$ in~\eqref{eq:exprxi} allows us to compute $\Pr{\xi=k}$ in $e^{o(k)}$ steps for any $k \ge 0$, since we assumed $\Pr{\zeta = k}$ to be computable in $e^{o(k)}$ steps. Furthermore, using $\Ex{\zeta}=1$ it follows from~\eqref{eq:exprxi} that $\Ex{\xi}=1$, see~\cite[Thm.~6]{MR3335013} for details on the calculation.
Moreover, $\xi$ has finite exponential moments.
Thus, Algorithm~\ref{algo:size-constrained-gw-tree-devroye} samples from the distribution of $\mT_n$ in expected time $O(n)$.

We proceed with the analysis of Step (2) in Algorithm~\ref{algo:anomega}. Determining the maximum degree~$K$ takes $O(n)$ steps. As argued before, constructing the sampler $\Gamma_{\le K} \cR(t)$ takes $e^{o(K)}$ steps. Hence, the expected time for doing so is bounded by
\[
    \Ex{e^{o(K)}}
    \le
    \mathbb{E}\left[\sum_{v \in \mT_n} e^{o(d^+_{\mT_n}(v)}\right].\]
By identical arguments as for Equations~\eqref{eq:Ex_W_n},~\eqref{eq:W_n}, and~\eqref{eq:GG}, and since $\xi$ has an exponential tail, it follows that this bound belongs to $O(n)$.

The unique generating series $\cA_\cR(z)$ with $\cA_{\cR}(z) = z \cR(\cA_\cR(z))$ satisfies $[z^n] \cA_\cR(z) = \Pr{L_\Omega(\mA) = n}$.
By~\eqref{eq:rhoaaa} it follows that it has radius of convergence $\rho_{\cA_\cR} = 1$ and hence $\cA_\cR(\rho_{\cA_\cR})=\cR(1) = 1$ (since $\cR(z) = \Ex{z^\xi}$).  We observed above that $\rho_\cR>1$.
Thus, as justified in the proof of Thm.~\ref{thm:algo_runtime_n}, we may generate the decoration $\beta_n$ in expected time $O(n)$ by repeatedly running $\Gamma_{\le K} \cR(t)$ for each vertex $v \in \mT_n$ until we generate an element from $\cR_{d_{\mT_n}^+(v)}$. 

We conclude with the analysis of the blow-up procedure in Step (3). The time required for performing the blow-up of a vertex $v \in \mT_n$ with decoration $\beta_n(v) = (y(v), x_1(v), \ldots, x_{\ell(v)}(v))$ is bounded by $O(d_{\mT_n}^+(v) + \ell(v))$.
Recall that the outdegrees of a tree with $n$ vertices sum up to $n-1$. 
Summing over the $n$ vertices of $\mT_n$, the total time for performing all blow-up operations is hence bounded by 
\[
	n-1 + \sum_{v \in \mT_n} \ell(v).
\]
Arguing analogously as for Equations~\eqref{eq:Ex_W_n},~\eqref{eq:W_n}, and~\eqref{eq:GG}, only with conditional moments, and using $t>1$ and $\Ex{\ell}<\infty$, it follows that
\begin{align*}
	\Ex{\sum_{v \in \mT_n} \ell(v)} &= O(n) \sum_{k \ge 0} \Ex{\ell \mid y + x_1 + \ldots + x_\ell= k }  \Pr{\xi=k} \\
	&= O(n) \Ex{\ell} \sum_{k \ge 0} \frac{\Pr{\xi=k}}{\Pr{y + x_1 + \ldots + x_\ell = k}} \\
	&= O(n) \sum_{k \ge 0} \frac{[z^k] \cR(z)}{[z^k] \cR(tz)} \\
	&= O(n).
\end{align*}
It follows that the expected time for performing the blow-up is $O(n)$. Hence the total expected time for generating $\mA_n^\Omega$ using Algorithm~\ref{algo:anomega} is $O(n)$.
\end{proof}

\begin{remark}
	\label{rem:conditions}
	Throughout, we assumed that $0 \in \Omega$. Rizzolo's~\cite{MR3335013} methods may be used to generalize the procedure so that this assumption is no longer necessary. However, the decorations and the blow-up procedure are far more technical in the case $0 \notin \Omega$. We leave the details  to the reader, because all applications of the present section to models of combinatorial structures considered below are already covered by the special case $\Omega= \{0\}$.
\end{remark}

\subsubsection{Dissections of convex polygons}
Let $\cD = \cX + \Seq_{\ge2}\circ\cD$ denote the class of dissections of polygons, compare to Figure~\ref{fi:decomp1}. 
In this section we present a sampler generating uniform dissections~$\mD_n$ of size $n$ in expected time $O(n)$ as an application of Algorithm~\ref{algo:anomega} for $\mA_n^\Omega$ with $\Omega=\{0\}$. First we need some notation and an alternative viewpoint for the class $\cD$. %For our purposes we equivalently see dissections as follows.
For $n\ge 3$ let $P_n$ denote the polygon in the complex plane with $n$ sides whose vertices are the $n$-th roots of unity. A dissection $D$ of $P_n$ is the union of all sides of $P_n$ together with a collection of diagonals (connecting vertices of $P_n$) that may only intersect in their endpoints. Then $\cD_n$ contains all dissections of $P_{n+1}$. See the first two images in Figure~\ref{fi:polygon_tree} for an example.
% \begin{figure}[h]
% 	\centering
% 	\begin{minipage}{0.5\textwidth}
%   		\centering
%   		\includegraphics[width=1.0\textwidth]{polygon_dissection.pdf}
% 	\end{minipage}
% 	\caption{The polygon $P_8$ and a possible dissection.}
%   	\label{fi:polygon}
% \end{figure}

\begin{figure}[h]
	\centering
	\begin{minipage}{0.8\textwidth}
  		\centering
  		\includegraphics[width=1.0\textwidth]{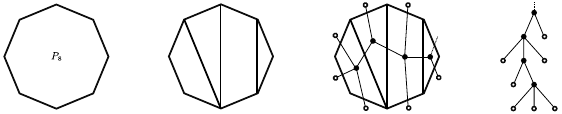}
	\end{minipage}
	\caption{From left to right: the polygon $P_8$, a dissection in $\cD_7$, the bijection $\Phi$ applied to this dissection and the corresponding tree.}
  	\label{fi:polygon_tree}
\end{figure}

In order to specify a sampler for $\mD_n$ (the uniform distribution on $\cD_n$) we exploit the following fact that can be found for instance in \cite[Prop.~2.2]{MR3245291}. 
There is a bijection $\Phi$ between $\cD_n$ and rooted plane trees with $n$ leaves where no node has outdegree $1$.
We abbreviate this set of trees by $\cH_n$.
Given a dissection $D\in\cD_n$ the tree $\Phi(D)\in\cH_n$ is constructed as follows, see also the transition from the third to the fourth image in Figure~\ref{fi:polygon_tree}.
First place a vertex in each of the faces of $D$ and outside each side of $P_{n+1}$. Then join any two vertices whose corresponding faces share a common edge. Let the root be the vertex connected to the vertex outside of the side connecting vertex $1$ with $\mathrm{e}^{2\pi \mathrm{i}/(n+1)}$ and delete this vertex and its adjacent edge. Vice versa given a tree $T\in\cH_n$ it is straightforward to obtain the corresponding dissection $\Phi^{-1}(T)\in\cD_n$ as depicted in the example (third and fourth image) in Figure~\ref{fi:polygon_tree}. 

It was shown in \cite{MR3245291} that there is a model of Galton-Watson-trees in $\cH_n$ corresponding to uniform dissections.
More concretely, for $c\in(0,1/2)$ consider the distribution of a random variable $\zeta$ given by
\[
    \Prb{\zeta=0} = \frac{1-2c}{1-c},
    \quad \Prb{\zeta=1}=0 \quad\text{and}\quad
    \Prb{\zeta=k} = c^{k-1}~~\text{for}~~ k\ge 2.
\]
Denote by $\mA_n$ the Galton-Watson tree with offspring distribution $\zeta$ and being conditioned on having $n$ leaves. Note that this corresponds to $\mA_n^\Omega$ in the previous section in the special case $\Omega=\{0\}$. Then $\mA_n$ has the same distribution as $\Phi(\mD_n)$ for any $c\in(0,1/2)$ according to \cite[Prop.~2.3]{MR3245291}. With this at hand, the sampler for $\mD_n$ involves two steps.
\begin{algo} Uniform dissection $\mD_n$ from $\cD_n$.
\label{algo:unif_dissection}
\begin{enumerate}
    \item Use Algorithm~\ref{algo:anomega} to generate $\mA_n$.
    \item Translate $\mA_n$ to $\Phi^{-1}(\mA_n)$.
\end{enumerate}
\end{algo}
In Step~$(1)$ of Algorithm~\ref{algo:unif_dissection} we generate a degree sequence $(d_1,\dots,d_K)$ for some $K\in\ndN$ representing the rooted plane tree $\mA_n$ with $n$ leaves (and $K+1$ vertices), compare to~\eqref{eq:condition_degree_sequence_GW}. To state a complete sampler for $\mD_n$ we still need to clarify a subroutine for Step~$(2)$ translating this sequence into the corresponding dissection.
% , see Figure~\ref{fi:degree_to_dissection} for a visualization.
\begin{algo} Translating a degree sequence $(d_1,\dots,d_K)$ to the corresponding dissection.
\label{algo:bijection_dissection}
\begin{enumerate}
    \item Create a directed cycle with $d_1+1$ vertices labelled counterclockwise by $\{1,\dots,d_1+1\}$. Let $E_1 = ((1,2),(2,3),\dots,(d_1+1,1))$ be the sequence of edges.
    \item For $2\le i\le K$ set $E_i=E_{i-1}$ if $d_i=0$. If $d_i>0$ (implying that $d_i\ge 2$) do the following. Let $(v_1,v_2)$ be the $i$-th edge in the sequence $E_{i-1}$. Create a directed cycle of size $d_i+1$ such that one edge is $(v_1,v_2)$. Label the remaining vertices counterclockwise (starting at $v_2$) with successive labels in $\ndN$ which have not been used so far. Append the newly created edges counterclockwise to $E_{i-1}$ to obtain $E_i$.
    \item Let $V$ contain all the labels in $E_K$ and let $E$ be the set of edges in $E_K$. Return the graph $(V,E)$. (To draw the dissection in the way defined above embed the graph $(V,E)$ into the complex plane such that its vertices are the $n$-th roots of unity, the vertex with label $1$ sits at $1$ and all the edges are non-crossing. Drop the labels afterwards.)
\end{enumerate}
\end{algo}

% \begin{figure}[h]
% 	\centering
% 	\begin{minipage}{1.0\textwidth}
%   		\centering
%   		\includegraphics[width=1.0\textwidth]{degree_to_dissection.pdf}
% 	\end{minipage}
% 	\caption{Building a dissection from the tree associated with the degree sequence $(2,2,0,0,2,0,0)$ with Algorithm~\ref{algo:bijection_dissection} -- \textcolor{red}{leider zu kleine Schrift.. Lohnt es sich das anzupassen oder eh unn\"otig?}}
%   	\label{fi:degree_to_dissection}
% \end{figure}
\begin{theorem}
Choosing $c=1-2^{-1/2}$ Algorithm~\ref{algo:unif_dissection} has expected runtime $O(n)$.
\end{theorem}
\begin{proof}
Per definition of $\zeta$ we have that $\Prb{\zeta=0}>0$ and $\Prb{\zeta\ge 2}>0$ verifying~\eqref{eq:cond_zeta}. The choice $c=1-2^{-1/2}$ further guarantees that $\Exb{\zeta}=1$ and for $\eps>0$ such that $(1+\eps)c<1$ we have that
\[
    \Exb{(1+\eps)^\zeta}
    = (1+\eps)\frac{1-2c}{1-c} +c^{-1} \sum_{k\ge 2}(1+\eps)^kc^k
    <\infty.
\]
Hence Equation~\eqref{eq:finexpzeta} is valid. Finally, the probability $\Prb{\zeta=k}$ can be computed in $\mathrm{e}^{o(k)}$ steps as $c$ is explicitly given. We deduce that all conditions at the beginning of Section~\ref{sec:gw_leaves} are fulfilled so that Algorithm~\ref{algo:anomega}, Step~$(1)$ of Algorithm~\ref{algo:unif_dissection} respectively, runs in expected time $O(n)$. 

Let us next explain why Step~$(2)$ or equivalently Algorithm~\ref{algo:bijection_dissection} runs in expected time $O(n)$ for any degree sequence $(d_1,\dots,d_K)$ corresponding to a tree $T\in\cH_n$. First of all note that there is no vertex with outdegree $1$ in $T$ implying that $K \le 2n$ so that we have $O(n)$ iterations in steps $(1)$ and $(2)$ of  Algorithm~\ref{algo:bijection_dissection}. 
The number of created edges is $n+1$ (for the edges of the polygon of length $n+1$) plus the additional edges accounting for the diagonals of the dissection. But in each iteration of Step~$(2)$ at most one diagonal edge (if $d_i>0$) is created so that the total number of edges is $O(n+1+K)=O(n)$. As each edge is only created once the total time needed to finish Step~$(2)$ is $O(n)$.
\end{proof}
\subsection{Subcritical substitution-closed classes of permutations}
An $n$-sized permutation  $\sigma: [n] \to [n]$  may be denoted in multiple ways, for example by the sequence of numbers $\sigma(1)\sigma(2)\ldots\sigma(n)$, or graphically by a diagram corresponding to  the collection of points $\{ (i, \sigma(i)) \mid i \in [n] \}$. Given permutations $\nu_1, \ldots, \nu_n$ of arbitrary sizes $k_1, \ldots, k_n \ge 1$, we may form the $(k_1 + \ldots + k_n)$-sized permutation $\sigma[\nu_1, \ldots, \nu_n]$ by performing a \emph{substitution}-operation, where for each $1 \le i \le n$ the point $(i, \sigma(i))$ gets replaced by the diagram of the permutation $\nu_i$, and the rows and columns are rescaled accordingly. This is best explained by Figure~\ref{fi:substperm} which depicts an example.

\begin{figure}[h]
	\centering
	\begin{minipage}{1.0\textwidth}
		\centering
		\includegraphics[width=0.6\textwidth]{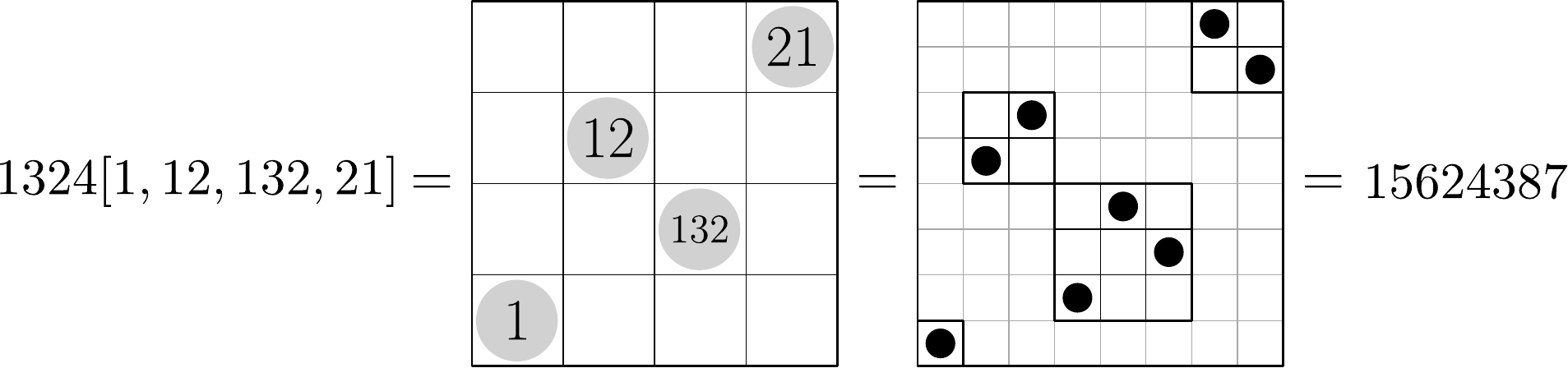}
		\caption{The substitution $\sigma[\nu_1, \ldots, \nu_4]$ for $\sigma= 1324$, $\nu_1 = 1$, $\nu_2 = 12$, $\nu_3= 132$, and $\nu_4 = 21$.}
		\label{fi:substperm}
	\end{minipage}
\end{figure}

A class $\cC$ of permutations is called \emph{substitution-closed} if $\sigma, \nu_1, \ldots, \nu_n \in \cC$  implies $\sigma[\nu_1, \ldots, \nu_n] \in \cC$. Letting $\cC_n \subset \cC$ denote the subset of $n$-sized permutations in $\cC$, the ordinary generating function of $\cC$ is given by $\cC(x) = \sum_{n \ge 1} |\cC_n| x^n$. A permutation $\sigma$ of size at least $3$ is called \emph{simple}, if it cannot be represented as a substitution of permutations, except of course in a trivial manner by $\sigma = \sigma[1, \ldots, 1]$ and $\sigma= 1[\sigma]$. We let $\cS \subset \cC$ denote the subclass of permutations that are simple and lie in $\cC$. 

\begin{definition}
We call the substitution-closed class $\cC$ of permutations \emph{subcritical}, if the radius of convergence $\rho_\cS$ of the ordinary generating series $\cS(z)$ satisfies
\begin{align}
	\label{eq:subcrit}
	\cS'(\rho_\cS) > \frac{2}{(1 + \rho_\cS)^2}-1.
\end{align}
\end{definition}

This is always satisfied if $\rho_\cS= \infty$, for which the right-hand side equals $-1$ per convention. In particular, it encompasses the case when $\cS$ is finite. The subclass $\cS$ of simple permutations plays an  analogous role for substitution-closed classes of permutations as the class of blocks does for block-stable classes of graphs. Inequality~\eqref{eq:subcrit} is the analogon to the subcriticality condition  in Definition~\eqref{def:subcritgraph} for subcritical classes of graphs. The condition also crops up in work on permutron limits~\cite{zbMATH07286839,10.1214/20-EJP469}.

If a Boltzmann sampling procedure is available for the subclass $\cS$ of simple permutations (for example, when this class if finite), and if the class $\cC$ is subcritical in the sense defined above, then a Boltzmann sampler for $\cC$ may be constructed that lets us generate a uniform $n$-sized permutation from $\cC$ in expected time $O(n^2)$. In the following, we pursue a different approach that allows us to perform this task in time $O(n)$ instead, assuming that we may compute the number $[x^k]\cS(x)$ of simple permutations of size $k$ in $\cC$ in at most $e^{o(k)}$ steps.

For all $k \ge 2$ we define the permutations $\oplus_k = 1\ldots k$ and $\ominus_k = k \ldots 1$.
We call a permutation \emph{$\oplus$-indecomposable} if it cannot be expressed as a substitution $\oplus_k(\nu_1, \ldots, \nu_k)$ for some $k \ge 2$.
The term \emph{$\ominus$-indecomposable} is defined analogously.
As detailed in \cite[Prop.~2]{albert2005simple}, any permutation in $\cC$ of size at least $2$ may be uniquely represented as a substitution $\sigma[\nu_1, \ldots, \nu_n]$ where $\nu_1, \ldots, \nu_n \in \cC$ and exactly one of the following three cases hold:
\begin{enumerate}[label=(\alph*)]%[\qquad a)]
	\item $\sigma \in \cS$, or
	\item $\sigma \in \{ \oplus_k \mid k \ge 2\}$ and $\nu_1,\ldots, \nu_k$ are $\oplus$-indecomposable, or
	\item $\sigma \in \{ \ominus_k \mid k \ge 2\}$ and $\nu_1,\ldots, \nu_k$ are $\ominus$-indecomposable.
\end{enumerate}
\begin{comment}
Letting $\cX$ denote the class containing only the unique permutation of size $1$, this may be expressed by the following decomposition:
\begin{align*}
	\cC &= \cX + \cS(\cC) + \cC_{\oplus} + \cC_{\ominus}, \\
	\cC_{\oplus} &= \Seq_{\ge 2}(\cX + \cS(\cC) +  \cC_{\ominus}), \\
	\cC_{\ominus} &= \Seq_{\ge 2}(\cX + \cS(\cC) + \cC_{\oplus} ).
\end{align*}
\end{comment}
This leads to unique representations of permutations from the class $\cC$ as canonical decomposition trees, which are plane trees  whose inner vertices are decorated with permutations from $\cS \cup \bigcup_{k \ge 2} \{\oplus_k, \ominus_k\}$, such that no two adjacent vertices may carry both an $\oplus$-decoration or both an $\ominus$-decoration. If the permutation is of the form $\sigma[\nu_1, \ldots, \nu_n]$ as in one of the three discussed cases, then the root of the associated tree is decorated with $\sigma$, and its $n$ children are roots of the (recursively defined) canonical decomposition trees corresponding to the permutations $\nu_1, \ldots, \nu_k$. See Figure~\ref{fi:exct} for an illustration. This way, the size of the permutation corresponds to the number of leaves of the associated canonical decomposition tree. 

\begin{figure}[h]
	\centering
	\begin{minipage}{1.0\textwidth}
		\centering
		\includegraphics[width=0.6\textwidth]{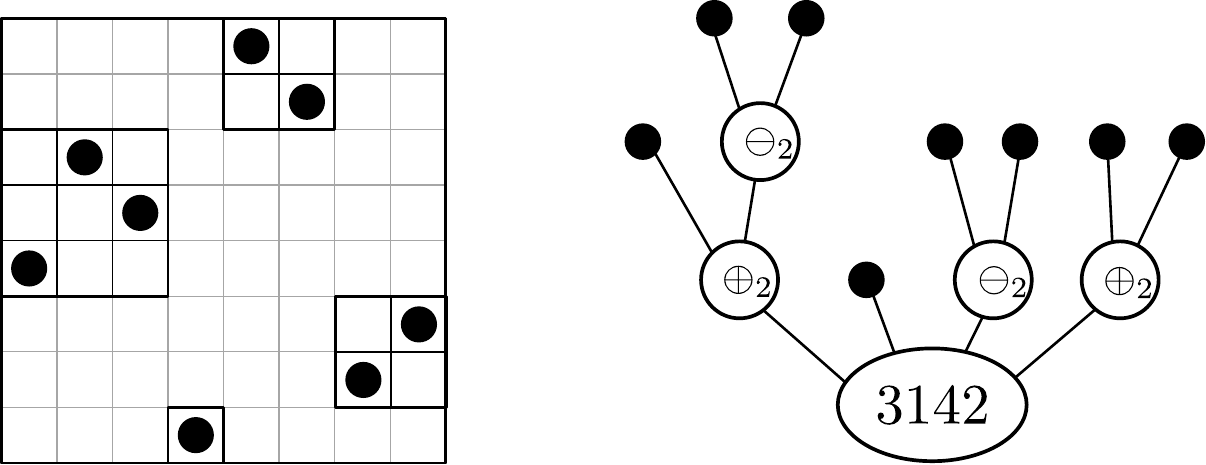}
		\caption{The permutation $46518723$ and its associated canonical decomposition tree.}
		\label{fi:exct}
	\end{minipage}
\end{figure}

Our first observation describes how the permutation associated to a canonical decomposition tree may be computed in linear time.

\begin{lemma}
	\label{le:permperm}
	A canonical decomposition tree with $n$ leaves may be transformed into a permutation in $O(n)$ steps.
\end{lemma}
\begin{proof}
A canonical decomposition tree with $n$ leaves is given by a plane tree $T$ with $n$ leaves, together with a family of permutations $(\sigma_v)_{v}$ from $\cS \cup \bigcup_{k \ge 2} \{\oplus_k, \ominus_k\}$ with the index $v$ ranging over all inner vertices of $T$. Of course, $T$ and $(\sigma_v)_v$ are subject to the discussed constraints, so that the size of $\sigma_v$ is equal to the number of children of $v$ for each inner vertex $v$, and so that no two adjacent inner vertices of $T$ are both decorated with a $\oplus$-permutation, or both with a $\ominus$-permutation. 

\subsubsection*{First step: assign labels to the leaves.} Note that any inner vertex of $T$  has at least two children, because all permutations from $\cS \cup \bigcup_{k \ge 2} \{\oplus_k, \ominus_k\}$ have size at least $2$. Hence the total number of vertices in such a tree with $n$ leaves is at most $2n-1$. Hence we may assign numeric labels from $1$ to $n$ to the leaves of $T$ according to their lexicographic order in $O(n)$ time by performing a depth-first-search traversal.
 
\subsubsection*{Second step: recursively calculate a linked list of numbers.} The next step is to form a linked list that represents the inverse of the permutation that we want to compute. We use a data type that additionally has pointers to the first and last element of the list, so that we may concatenate two such lists in a bounded number of steps, regardless of their length. See~\cite[Ch.~2]{knuth1997} for details on this data structure.
%\leon{linked list = sequence? Beispiel f\"ur datatype mit pointern?}
%\ben{ \url{https://en.wikipedia.org/wiki/Linked_list} Weiß nicht, ob wir das nochmal erklären sollten.}

The algorithm works recursively: If the tree consists of a single vertex, we return its numeric label as a linked list of length $1$. If it is not, then the root is decorated with some permutation $\sigma$, and has some number $d \ge 2$ of children. The decorated fringe subtrees $T_1, \ldots, T_d$ corresponding to these children have disjoint leaf label sets, and calling the algorithm recursively for each returns linked lists $L_1, \ldots, L_d$. The inverse of a permutation may be computed in linear time, hence we may calculate the inverse $\sigma^{-1}$ of $\sigma$ in $O(d)$ steps. The employed datatype allows us to form the concatenation $L$ of $L_{\sigma^{-1}(1)}, \ldots, L_{\sigma^{-1}(d)}$ in that order in $O(d)$ steps.
 
Now, the number of steps required for this algorithm is $O(d)$ plus the number of steps required for the recursive calls to compute $L_1, \ldots, L_d$. Hence, each vertex of $T$ contributes an $O(d^+_T(v))$ number of steps, with $d^+_T(v)$ denoting its number of children. Hence the total number of steps required to compute the list $L$ is $O(\sum_{v \in T} d^+_T(v)) = O(n)$.

\subsubsection*{Third step: return the inverse of the permutation associated to that list.}

The list $L$ computed in the second step corresponds to a permutation that for each $1 \le i \le n$ maps the number $i$ to the $i$th element of $L$. The inverse $\sigma$ of this permutation may be calculated in $O(n)$ steps.

\subsubsection*{Correctness and time complexity}
We have argued that each of the three steps may be completed in $O(n)$ steps, hence the algorithm completes in linear time. In order to check that it actually computes the permutation associated to $(T, (\sigma_v)_v)$, simply note in the second step that if for each $1 \le i\le d$ the list $L_i$ represents the inverse of the permutation corresponding to the tree $T_i$, then the concatenation $L$ of $L_{\sigma^{-1}(1)}, \ldots, L_{\sigma^{-1}(d)}$ represents the inverse of the permutation corresponding to $T$. Hence correctness of the algorithm follows by structural induction.

\subsubsection*{A closing example}
Let us close with an example. The canonical decomposition tree $T$ in Figure~\ref{fi:exct} consists of an outdegree $d=4$ root vertex decorated by the permutation $\sigma= 3142$, with $4$ decorated trees $T_1, \ldots, T_4$ attached to it. The first step of the algorithm labels the leaves from $1$ to $8$ in lexicographic order, that is, from left to right in the drawing in Figure~\ref{fi:exct}. The lists corresponding to the subtrees in the second step are given by $L_1 = (1,3,2)$, $L_2 = (4)$, $L_3 = (6,5)$, and $L_4 = (7,8)$. The inverse of $\sigma$ is given by $\sigma^{-1} = 2413$. Hence the list $L$ is given by the concatenation of $L_2, L_4, L_1, L_3$, that is, $L = (4,7,8,1,3,2,6,5)$. Its inverse is the permutation $47813265$ corresponding to the tree~$T$.
\end{proof}

The drawback of canonical decomposition trees is that the constraints for the decoration of the children of a vertex depend on the decoration of the vertex itself. This violates one of the requirements of $\cR$-enriched trees, where the decoration of a vertex is only constrained by the number of its children.

For this reason, packed trees were introduced in~\cite{10.1214/20-EJP469}. The idea is to encode canonical decomposition trees by trees with different kinds of decorations. To this end, we define a \emph{gadget} as a special kind of canonical decomposition tree, with the additional requirements that it has height at most $2$, and the root is an internal vertex decorated by a simple permutation, and each child of the root is either a leaf or an internal vertex decorated by an increasing permutation from $\{\oplus_k \mid k \ge 2\}$. The size of a gadget is its number of leaves. We define the class $\cQ$ as the union of the collection of all gadgets and the collection  $\{ \circledast_k \mid k \ge 2\}$ of formal objects, the index $k$ denoting the formal size of such an object $\circledast_k$. Thus, the ordinary generating series of the class $\cQ$ is given by
\begin{align}
	\cQ(x) = \frac{x^2}{1-x} + \cS \left( \frac{x}{1-x} \right).
\end{align}

A \emph{packed tree} is a  rooted plane tree where each internal vertex is decorated by an object from the class $\cQ$, with the size of the object matching the number of children of the vertex. The size of a packed tree is defined to be its number of leaves.

\begin{figure}[t]
	\centering
	\begin{minipage}{1.0\textwidth}
		\centering
		\includegraphics[width=0.8\textwidth]{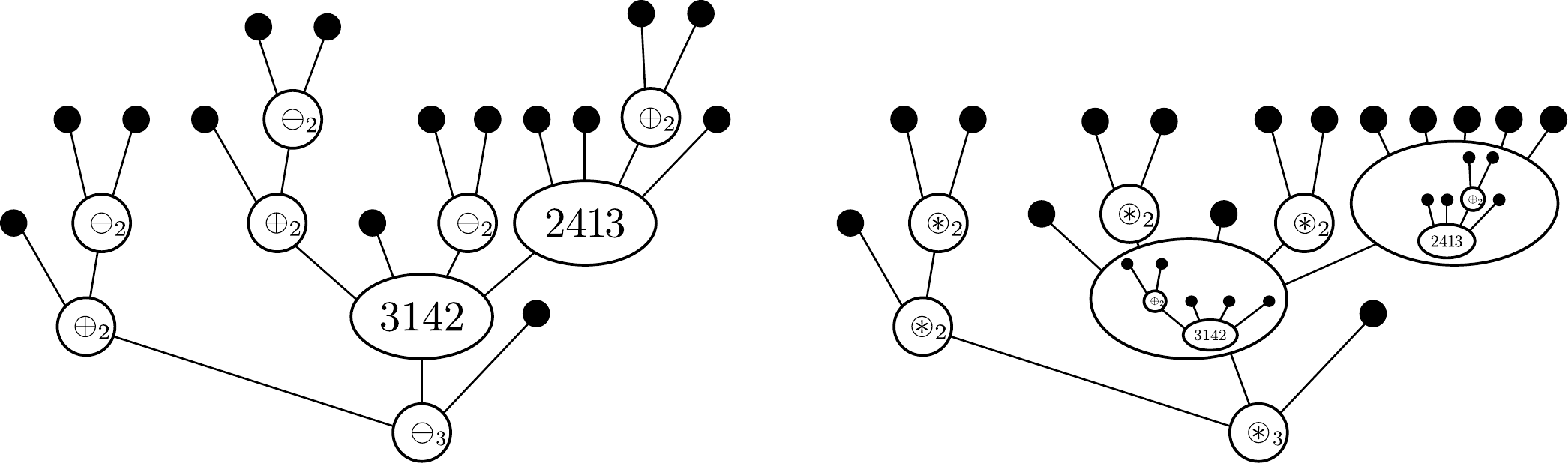}
		\caption{A canonical decomposition tree and its corresponding packed tree.}
		\label{fi:packed}
	\end{minipage}
\end{figure}

As argued in~\cite[Prop. 2.15]{10.1214/20-EJP469}, there is a size-preserving bijection between the class of packed trees and the subclass of canonical decomposition trees whose root is not decorated by an increasing permutation from $\{\oplus_k \mid k \ge 2\}$.  That is, those that correspond to $\oplus$-indecomposable permutations in $\cC$. Let us call such decomposition trees $\oplus$-indecomposable, and define $\ominus$-indecomposable decomposition trees analogously.

Our next observation tells us that the number steps required for applying this bijection is linear in the size of the input.

\begin{lemma}
	\label{eq:packi}
	The $\oplus$-indecomposable canonical decomposition tree corresponding to a given packed tree with $n$ leaves may be computed in $O(n)$ steps.
\end{lemma}
\begin{proof}

The canonical decomposition tree associated to a packed tree $P$ gets constructed in two steps:

\subsubsection*{First step: Blow-up gadgets.} We perform a blow-up procedure where  each internal vertex $v$ of $P$ that is decorated by a gadget $G$ gets replaced by its gadget. That is, we delete $v$ and add an edge between its  parent and the root of $G$. For each integer $i$ from $1$ to the number $d$ of  children of $v$ we  merge the root of the $i$-th subtree attached to $v$ with the $i$-th leaf of $G$.

The time required to perform a single blow-up of $v$ is $O(d)$. Hence the total time required for all blow-ups is linear in the number of vertices of $P$. Since $P$ has $n$ leaves and each internal vertex has at least two children, its number of vertices is at most $2n-1$. Hence the time required for the first step is $O(n)$.
	 
\subsubsection*{Second step: Replace $\circledast$-signs.} 	
We traverse the vertices of the tree $\tilde{P}$ resulting from the first step in a breadth-first-search order. Whenever we encounter a vertex $v$ that is decorated by $\circledast_k$ for some $k \ge 2$, we replace $\circledast_k$ by either $\ominus_k$ or $\oplus_k$ according to the following rule. If $v$ is the root of $\tilde{P}$, then we replace it with $\ominus_k$. If $v$ is not equal to the root of $\tilde{P}$, then its parent is decorated with $\ominus_k$ or $\oplus_k$ (possibly due to modifications done in prior steps of the breadth-first-search traversal), and we replace the decoration of $v$ with the opposite sign of its parent.

The time required for this is linear in the number of vertices of $\tilde{P}$. As $\tilde{P}$ has $n$ leaves and each internal vertex has at least two children, it has at most $2n-1$ vertices. Hence the time required for the second step is $O(n)$. 

\end{proof}

Note that packed trees also correspond bijectively to $\ominus$-indecomposable canonical decomposition trees.  The only difference to the bijection with $\oplus$-indecomposable permutations here is that in the second step we replace the decoration of the root by $\oplus_k$ in case that it previously carried $\circledast_k$ for some $k \ge 2$.

Next, we are going to describe how uniform packed trees may be generated in expected linear time. The class $\cP$ of packed trees (with leaves as atoms) is related to the class $\cQ$ via
\begin{align}
	\cP(x) = x + \cQ(\cP(x)).
\end{align}
This identifies the class $\cP$ as  $\cQ$-enriched 
Schr\"oder
parenthesizations, which are in bijection to $\Seq(\cQ/\cX)$-enriched trees by the Ehrenborg--M\'endez bijection~\cite{MR1284403}. (Here $\cQ/\cX$ denotes the class constructed from $\cQ$ by shifting the sizes of the objects so that its generating series is given by $\cQ(x)/x$.) %\leon{Ich habe $x$ durch $\cX$ ersetzt wenn es um komb. Klassen geht. Okay?} 
We hence have two options for generating them in expected linear time. The first is to employ this bijection and apply our main Algorithm~\ref{algo:R-enriched_tree} for $\cR= \Seq(\cQ/\cX)$. The second option is to employ a variant of this algorithm that uses Galton--Watson trees conditioned on their number of leaves instead, as opposed to the total number of vertices. This is possible, since Algorithm~\ref{algo:anomega} allows us to sample these trees in expected linear time. From our viewpoint it is more natural to go with the second option.

Throughout the rest of this section let us assume that the class $\cC$ is subcritical, that is, Inequality~\eqref{eq:subcrit} holds. We let $\rho_\cP$ denote the radius of convergence of $\cP(x)$. It was shown in~\cite[Proof of Prop. 3.5 and Sec. 3.1]{10.1214/20-EJP469} that in this case $\cQ(x)$ has a positive radius of convergence $\rho_\cQ>0$, and that there is a unique number $0<\kappa \le \rho_\cS$ with 
\begin{align}
\cS'(\kappa) = 2/(1+\kappa)^2 -1.	
\end{align} Moreover, it was shown that  $y:= \cP(\rho_\cP)  = \kappa/(1+\kappa) \in ]0, \rho_\cS[$ satisfies $\cQ'(y) = 1$, and we may define a random non-negative integer $\zeta$\footnote{The constant $y$ and the random variable $\zeta$ defined here correspond to the constant $t_0$ and the random variable $\xi$ defined in~\cite[Prop. 3.5]{10.1214/20-EJP469}.} with probability generating function
\begin{align}
	\Ex{x^\zeta} = 1 - \cQ(y)/y + \cQ(x y) / y
\end{align}
having radius of convergence strictly larger than $1$, so that $\zeta$ has finite exponential moments. It follows that $\zeta$ satisfies
\begin{align}
	\Ex{\zeta} = \cQ'(y) = 1 .
\end{align}
As shown in~\cite[Sec. 6.4]{MR4132643}  for general enriched Schr\"oder parenthesizations, and noted in~\cite[Lem. 3.4, Sec. 3.4]{10.1214/20-EJP469} for the special case of the present setting, a uniform packed tree may be generated by conditioning a $\zeta$-Galton--Watson tree on having $n$ leaves, and adding uniform $\cQ$-decorations to each internal vertex in the second step.
%\leon{fehlt hier nicht noch, dass $\Prb{\zeta=k}$ in $\exp{o(k)}$ steps berechnet werden kann? Sollten wir im Allgemeinen besser hervorheben, unter welchen Bedinungen $O(n)$ erreicht wird? Wir brauchen ja auch, dass $\Gamma Q(t_0)$ in konstanter Zeit l\"auft.}
 The result of the present work enable us to perform this in expected linear time, in analogy to Algorithm~\ref{algo:R-enriched_tree}.
Suppose that for some $t_0>0$ with $\cP(\rho_\cP) < t_0 < \rho_\cS$ there is a Boltzmann sampler $\Gamma \cQ(t_0)$ that runs in expected finite time. Recall that we assume that $[x^k]\cS(x)$ (and hence  also $\Pr{\zeta=k}$) may be computed in $\exp(o(k))$ steps. 

\begin{algo}	\label{eq:algoparenthesizations}  A generator for a uniform packed tree with $n$ leaves that runs in expected time $O(n)$.
	\begin{enumerate}
		\item  Generate  a $\zeta$-Galton--Watson tree $\mA_n$ conditioned  having $n$ leaves using Algorithm~\ref{algo:anomega} for the special case $\Omega = \{0\}$.
		\item For each internal vertex $v$ of $\mA_n$ let $d \ge 2$ denote its number of children and select a uniform $d$-sized $\cQ$-decoration by repeatedly running the Boltzmann sampler $\Gamma \cQ(t_0)$ until it produces a $d$-sized object.
	\end{enumerate}
\end{algo} 
Here is a justification why this runs in expected linear time:
\begin{proof}
The first step terminates in expected time $O(n)$ as shown in Algorithm~\ref{algo:anomega}. The fact that adding decorations using a Boltzmann sampler above the critical threshold $\cP(\rho_\cP)$ also takes expected time $O(n)$ may be verified by recalling that the $\zeta$-Galton--Watson tree $\mA_n$ with $n$ leaves is constructed in Algorithm~\ref{algo:anomega} from an associated $\xi$-Galton--Watson tree with $n$ vertices, enabling us to adapt the arguments of the proof of Algorithm~\ref{algo:R-enriched_tree} in a straight-forward way.
%where each internal vertex is decorated with a sequence $(x_1, \ldots, x_\ell)$, $\ell \ge 0$ of non-negative integers,  drawn with probability proportional to the weight $\prod_{i=1}^\ell \Pr{\zeta = x_i +1}$. Hence decorating the tree $\mA_n$ using $\Gamma \cQ(t_0)$ is equivalent to decorating $\mT_n$ so that a vertex with primary decoration $(x_1, \ldots, x_\ell)$ receives a secondary decoration $(Q_1, \ldots, Q_\ell)$ such that $Q_i$ is an $(x_i +1)$-sized $\cQ$-object obtained by repeatedly calling  $\Gamma \cQ(t_0)$ until obtaining an object with this size. The expected time for this is the expectation of $\ell$ geometric random variables, yielding a total  waiting time of $\ell \cQ(t_0) / (q_k t_0^k)$.
\end{proof}

Recall that a packed tree corresponds bijectively to an $\ominus$-indecomposable canonical decomposition tree, and likewise to an $\oplus$-indecomposable canonical decomposition tree. 

\begin{algo}	\label{eq:permsampler}  A generator for a uniform $n$-sized permutation from the subcritical class $\cC$ that runs in expected time $O(n)$.
	\begin{enumerate}
		\item  Use Algorithm~\ref{eq:algoparenthesizations} twice to sample two independent $n$-sized packed trees $P_1$ and $P_2$. If both have a root decorated by $\circledast$-symbols, discard them and try again until at least one of the two has a root decorated by a gadget.
		\item We make a case distinction. 
%		\leon{Notation: $\mathsf{P}$ statt $P$?} \ben{Bin mir nicht sicher, das würde so wirken, als wären das uniform packed trees der Größe 1 und 2. :-/}
		\begin{enumerate}
			\item If both $P_1$ and $P_2$ have a root decorated by a gadget, use the procedure from Lemma~\ref{eq:packi} to compute the canonical decoration tree $T$ corresponding to $P_1$.
			\item If $P_1$ has a root decorated by a gadget, but $P_2$ doesn't, then use the procedure from Lemma~\ref{eq:packi} to compute the canonical decoration $\ominus$-indecomposable tree $T$ corresponding to $P_2$.
			\item If $P_2$ has a root decorated by a gadget, but $P_1$ doesn't, then use (minor adaption of) the procedure from Lemma~\ref{eq:packi} to compute the canonical decoration $\oplus$-indecomposable tree $T$ corresponding to~$P_1$.
		\end{enumerate}
		\item Use the procedure from Lemma~\ref{le:permperm} to compute the permutation corresponding to the canonical decomposition tree $T$. Return this permutation.
	\end{enumerate}
\end{algo} 

\begin{proof}
	First, let us verify that this algorithm actually samples a uniform permutation from $\cC$. The generating series $\cT(x)$ for canonical trees (identical to $\cC(x)$) may be split up into three series,
	\[
		\cT(x) = \cT_{\oplus}(x) + \cT_{\ominus}(x) + \cT_S(x),
	\]
	depending on whether the root is decorated with an $\oplus$-symbol, and $\ominus$-symbol, or a simple permutation from $\cS$. By symmetry it holds that
	\[
		\cT_{\oplus}(x) = \cT_{\ominus}(x).
	\]
	A packed tree whose root is decorated with a $\circledast$-symbol may be interpreted either as an element from $\cT_{\oplus}$ or from $\cT_{\ominus}$. A packed tree whose root is decorated by a gadget corresponds to a canonical decoration tree from $\cT_S$.

	Hence $a_n = [x^n] \cT_{\oplus}(z) = [x^n] \cT_{\ominus}(x)$ equals the number of packed trees with a $\circledast$-root, and $b_n = [x^n] \cT_S(x)$ equals the number of packed trees with a gadget root. The canonical tree $T$ generated in the first two steps of the procedure satisfies 
	\[
		\Pr{T \in \cT_S} = \left(\frac{b_n}{a_n + b_n}\right)^2 \left(1- \left(\frac{a_n}{a_n + b_n}\right)^2\right)^{-1} = \frac{b_n}{2a_n + b_n}.
	\]
	Likewise, 
	\[
		\Pr{T \in \cT_\ominus} = \Pr{T \in \cT_\oplus} = \frac{a_n b_n}{(a_n + b_n)^2} \left(1- \left(\frac{a_n}{a_n + b_n}\right)^2\right)^{-1} = \frac{a_n}{2a_n + b_n}.
	\]
	Conditional on belonging to either of these three classes the tree $T$ is uniformly distributed. Hence $T$ is uniformly distributed among all canonical decomposition trees with $n$ leaves. Consequently, the Algorithm produces a uniform $n$-sized permutation from the class $\cC$.
	
	As for the performance of this algorithm, note that the number of pairs $(P_1, P_2)$ we need to sample in the first step follows a geometric waiting time for an event with probability  $1- \left(\frac{a_n}{a_n + b_n}\right)^2$. By for example~\cite[Eq. (12), (13)]{{10.1214/20-EJP469}} we know that $c_n := [x^n] \cT(x) = 2 a_n +b_n$ satisfies $c_n \sim (a_n +b_n) / (1 - \cP(\rho_\cP))^2$ with $\cP(\rho_\cP) \in ]0,1[$. Dividing by $a_n +b_n$ on both sides it follows that
	\[
		\lim_{n \to \infty} \frac{a_n}{a_n + b_n} = \frac{1}{(1 - \cP(\rho_\cP))^2} -1 \in ]0,1[.
	\]
	Hence we need an expected finite number of pairs, each of which may be sampled with an expected time $O(n)$ by Algorithm~\ref{eq:algoparenthesizations}. Hence the first step takes time $O(n)$ in expectation. The time for the second step admits a deterministic $O(n)$ upper bound by Lemma~\ref{eq:packi}. Likewise, the time for the third step takes is $O(n)$ by Lemma~\ref{le:permperm}. Hence the expected time for the entire algorithm is $O(n)$.
\end{proof}

\subsection{Further examples}

As mentioned in the introduction, the framework of the present work allows for the construction of linear-time exact-size samplers for a large variety of classes. The key property are bijections between these classes to instances of $\cR$-enriched trees and resulting connections to mono-type branching processes. 

Going through the details would require us to recall large amounts of combinatorial background on these classes and their bijective encodings.  Since the construction of the samplers may be performed in analogous manner as in the treated examples, we will only briefly comment on each case. We also remark that the list presented in the present work makes no claim to be exhaustive. Further class might be treated in the same way.

\subsubsection{Outerplanar maps} Planar maps are embeddings of planar graphs into the $2$-sphere, considered up to orientation-preserving homeomorphism. The faces of a planar map are the connected components that remain after removing the map from the sphere. Usually one distinguishes and orients a root-edge, and calls the face to its right the outer face. A planar map is called outerplanar if all its vertices lie on the frontier of the outer face.
	
	As mentioned in the introduction, earlier work~\cite{MR2185278} already described a linear-time exact-size sampler for uniform random $n$-vertex simple outerplanar maps. Later,~\cite{aihpstufler2017} established a bijective encoding between outerplanar maps in terms of $\cR$-enriched trees for \[\cR= \Seq \circ \cD,\] with $\cD$ the class of dissections. Similar as for the block-stable graph classes treated in Subsection~\ref{ssec:sccg}, this class of $\cR$-enriched trees can be shown to be tame and the bijection may be applied in an expected linear time. This results in an alternative exact-size sampler that operates in expected linear time.
	
%	By restricting the class $\cD$ to bipartite dissections, we may also treat the family of bipartite outerplanar maps in an analogous way.

\subsubsection{Cographs}

Cographs may be characterized recursively as follows: Any graph consisting of a single vertex and no edges is a cographs. The disjoint union of two cographs is a cograph. The complement of cograph is a cograph.

The recent work~\cite[Lem. 5.1]{cographs} describes how a uniformly chosen cograph with $n$ labelled vertices may be generated from a tree $\tau_n$ obtained by conditioning a Bienaym\'{e}--Galton--Watson tree on having $n$ leaves. The offspring distribution $\zeta$ is given by its probability generating function 
\[
	\Ex{z^\zeta} = 2 (1 - 1 / \log 2) + 2^z / \log 2 - z.
\]
 A  parity $p \in \{\text{even}, \text{odd}\}$ is chosen uniformly at random (only once).
The $n$ leaves of the random tree $\tau_n$ form the vertex set of the associated cograph. Any two distinct leaves of the tree are adjacent in the cograph if the parity of the height of their lowest common ancestor in $\tau_n$ equals~$p$.

The tree $\tau_n$ may be generated in expected time $O(n)$ as described in Subsection~\ref{sec:gw_leaves}. The corresponding cograph may be generated in generated in $O(n^2)$ steps. The average runtime of the resulting exact-size sampling procedure is linear in the output size, since the expected number of edges of the random cograph has order $n^2$ by~\cite{cographs}.

\subsubsection{Level-$k$ phylogenetic networks}
	
Phylogenetic networks model the evolutionary history of species that have undergone reticulation events. From a mathematical perspective, they are simple rooted directed graphs with no directed cycles subject to the following constraints: The root has indegree $0$ and outdegree $2$. All non-root vertices are either tree nodes (indegree $1$, outdegree $2$), reticulation nodes (indegree $2$, outdegree $1$), or leaves (indegree $1$, outdegree $0$).

The leaves of a phylogenetic network are labelled by elements of a finite collection of species, similar to labels of a graph. Roughly speaking, given an integer $k \ge 1$ a level-$k$ network $N$ is a phylogenetic network where each block (of the associated undirected graph) contains at most $k$ reticulation nodes of $N$. Additionally, any block with at least $3$ vertices is required to contain at least $2$ vertices that are sources of bridges of $N$.

At least for $k=1$ and $k=2$, it was shown by~\cite{zbMATH07298368} that Boltzmann samplers may be constructed that allow approximate-size sampling of random level-$k$ networks in expected linear time, and exact-size sampling in expected quadratic time. Recent work~\cite[Sec. 2]{phylo} showed that for $k \ge 1$ random $n$-leaf phylogenetic networks may be generated by applying a blow-up procedure to a Bienaym\'{e}--Galton--Watson tree (with offspring distribution depending on $k$) conditioned on having $n$ leaves. This random tree may be generated in time $O(n)$ as described in Subsection~\ref{sec:gw_leaves}, and the expected time for applying the blow-up procedure is $O(n)$ by analogous arguments as in Subsection~\ref{ssec:sccg}. This results in a linear time exact-size sampler for random $n$-leaf level-$k$ phylogenetic networks.

\bibliographystyle{abbrv}
\bibliography{sampling}

\end{document}